\documentclass[10pt,a4paper,reqno]{amsart}
\usepackage{latexsym}
\usepackage{amsmath}
\usepackage[dvips]{graphicx}%
\usepackage{amsfonts}%
\usepackage{amssymb}

\newcommand{\CC}{\subset\hskip-0.2em\subset}
\newcommand{\beq}{\begin{equation}}
\newcommand{\eeq}{\end{equation}}

\def\R{\mathbb R}
\def\N{\mathbb N}

\def\Z{\mathbb Z}
\def\cal{\mathcal}
\def\Ht{{{\widetilde H}^1}}

\def\H{{\cal H}^{N-1}}

\def\de{\delta}
\def\e{\varepsilon}

\def\s{\sigma}
\def\vphi{\varphi}
\def\Om{\Omega}
\newcommand{\medint}{-\kern -,375cm\int}
\newcommand{\medintinrigo}{-\kern -,315cm\int}

\newcommand{\wto}{\rightharpoonup}

\def\Div{\operatorname*{div}\nolimits}

\newcommand{\T}{{{\mathbb T}^N}}
\renewcommand{\P}{P_\T}
\newcommand{\pa}{\partial}
\newcommand{\E}{\mathcal{E}}
\numberwithin{equation}{section}
\newtheorem{theorem}{Theorem}[section]

\newtheorem{corollary}[theorem]{Corollary}

\newtheorem{lemma}[theorem]{Lemma}

\newtheorem{proposition}[theorem]{Proposition}

\theoremstyle{definition}
\begingroup
\newtheorem{definition}[theorem]{Definition}
\newtheorem{remark}[theorem]{Remark}

\endgroup

\begin{document}

\title[Nonlocal isoperimetric problem]{Minimality via second variation\\for a nonlocal isoperimetric problem}

\author{E.Acerbi, N.Fusco, M.Morini}
\address[E.\ Acerbi]{Dipartimento di Matematica, Universit\`a degli Studi di Parma, Parma, Italy}
\email{emilio.acerbi@unipr.it}
\address[N.\ Fusco]{Dipartimento di Matematica e Applicazioni ``R. Caccioppoli'',
Universit\`{a} degli Studi di Napoli ``Federico II'' , Napoli, Italy}
\email{n.fusco@unina.it}
\address[M.\ Morini]{Dipartimento di Matematica, Universit\`a degli Studi di Parma, Parma, Italy}
\email{massimiliano.morini@unipr.it}

\dedicatory{Dedicated to Sergio Spagnolo on his 70th birthday}
\begin{abstract}
We discuss the local minimality of certain configurations for a nonlocal isoperimetric problem used to model microphase separation in diblock copolymer melts. We show that critical configurations  with positive second variation are local minimizers of the nonlocal area functional and, in fact, satisfy a quantitative isoperimetric inequality with respect to sets  that are  $L^1$-close.  The link with local minimizers for the diffuse-interface Ohta-Kawasaki energy is also discussed. As a byproduct of the quantitative estimate, we get  new results concerning periodic local minimizers of the area functional and a proof, via second variation, of the sharp quantitative isoperimetric inequality in the standard Euclidean case. As a further application, we address the global and local minimality of certain lamellar configurations.
\end{abstract}
\maketitle






\section{Introduction}
Diblock copolymers are extensively studied materials, used to engineer nanostructures thanks to their peculiar properties and rich pattern formation.
A well established theory used in the modeling of microphase separation for A/B diblock copolymer melts is based on the following energy first proposed by Ohta-Kawasaki, see \cite{OK}:
\beq\label{in.OK}
\mathcal E_\e(u):=\e\int_{\Om}|\nabla u|^2\, dx+\frac1\e\int_\Om(u^2-1)^2\, dx
+\gamma_0\int_\Om\int_\Om G(x,y)\bigl(u(x)-m\bigr)\bigl(u(y)-m\bigr)\, dx\,dy\;,
\eeq
where $u$ is an $H^1(\Om)$ phase parameter describing the density distribution of the components  ($u=-1$ stands for one phase, $u=+1$ for the other), $m=\medintinrigo_\Om u$ is the difference of the phases' volume fractions and $G$ is the Green's function for $-\Delta$. The parameter $\gamma_0\geq0$ is characteristic of the material.

Since $\e$ is a small parameter, from the point of view of mathematical analysis it is more convenient to consider the variational limit of the energy \eqref{in.OK}, which is given by
$$
\mathcal E(u):=\frac12 |D u|(\Om)+\gamma\int_\Om\int_\Om G(x,y)\bigl(u(x)-m\bigr)\bigl(u(y)-m\bigr)\, dx\,dy\;,
$$
where now $u$ is a function of bounded variation in $\Om$ with values $\pm 1$, $|Du|(\Om)$ is the total variation of $u$ in $\Om$, and $\gamma=3\gamma_0/16\geq0$. Writing
$$E=\{x\in\Om:u(x)=1\}\;,$$
so that $u=\chi_E-\chi_{\Om\setminus E}$, this energy may be rewritten in a useful geometric fashion as
\beq\label{in.OKgeom}
J(E)=P_\Om(E)+\gamma\int_\Om\int_\Om G(x,y)\bigl(u(x)-m\bigr)\bigl(u(y)-m\bigr)\, dx\,dy\;,
\eeq
where $P_\Om(E)$ is the perimeter of $E$ in $\Om$.

Competition between the short-range interfacial energy and the long-range nonlocal Green's function term in both functionals \eqref{in.OK} and \eqref{in.OKgeom} leads to pattern formation. Indeed the perimeter term drives the system toward a raw partition in few sets of pure phases with minimal interface area, whereas the Green's term is reduced by a finely intertwined distribution of the materials.

As observed in the literature, the domain structures in phase-separated diblock copolymers closely approximate periodic surfaces with constant mean curvature, see e.g.~\cite{TAHH}. Some of the most commonly observed structures are schematized in Figure~1.
\begin{figure}[h!]\label{figurauno}
\includegraphics[scale=0.8]{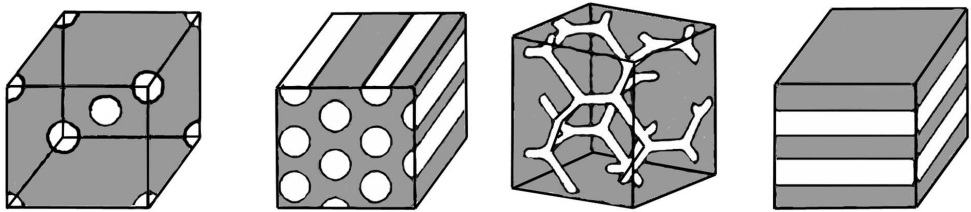}\vskip -0.2cm \caption{\small From left to right spherical spots, cylinders, gyroids and lamellae.}
\end{figure} 

A challenging mathematical problem is to prove that global minimizers of \eqref{in.OKgeom} are periodic: this is known to be true in one dimension, see e.g.~\cite{Mu,RW}, but still open in higher dimensions, where only partial results are known, see e.g. \cite{ACO,Sp}. We refer also to \cite{CP, CPW, CSp, vGP, GMS1, GMS2, KM1, KM2, Mur, ST, Topa} for other related results on global minimizers. A more reasonable task is to exhibit a class of periodic solutions which are local minimizers of the approximating and limit energies \eqref{in.OK} and \eqref{in.OKgeom}, rather than investigating general properties of global minimizers: this is the direction taken, among others, by Ren and Wei  and by Choksi and Sternberg.  The first authors in a series of papers \cite{RW0,RW1,RW2,RW3,RW4}  construct several examples of lamellar, spherical and cylindrical critical configurations and find conditions under which they are stable, i.e., their second variation  is positive definite.  The main contribution in  \cite{CS} is the computation of the second variation for general critical configurations of \eqref{in.OKgeom} (see also \cite{Mur0, MO}, where related linear stability/instability issues have been addressed for the first time, but a more physical perspective). However, all these papers leave open the basic question whether the positivity of the second variation implies local minimality. 

We give a full answer to this question by showing that any critical configuration of  \eqref{in.OKgeom} with positive definite
second variation is a local minimizer with respect to small $L^1$-perturbations. We now describe in more details the results proved here. We  consider both the periodic case, where $\Om=\T$ is the N-dimensional flat torus of unit volume, and the homogeneous Neumann case, where $\Om$ is a bounded smooth open set. We start by considering the periodic case. 

We recall that a sufficiently smooth critical set for $J$ satisfies the Euler-Lagrange equation
$$
H_{\pa E}(x)+4\gamma v(x)=\lambda \qquad\text{for all $x\in \partial E$,}
$$
where $H_{\pa E}(x)$ denotes the sum of the principal curvatures of $\partial E$ at $x$, the number $\lambda$ is a constant Lagrange multiplier associated to the volume constraint  $\int_\T u\, dx =m$, i.e., $|E|=(m+1)/2$, and
$$
v(x):=\int_\T G(x,y)(u(y)-m)\, dy
$$
is the unique solution to 
$$
-\Delta v=u-m \quad\text{in $\T$}\qquad \int_{\T}v\, dx =0\,.
$$
By the results of \cite{CS}, we can associate to the second variation of \eqref{in.OKgeom} at a regular critical  set $E$ the  quadratic
form $\pa^2J(E)$ defined over all functions $\vphi\in H^1(\pa E)$ such that $\int_{\pa E}\vphi\, d\H=0$ by
\begin{equation*}
\begin{split}
\pa^2J(E)[\varphi]& = \int_{\pa E}\bigl(|D_\tau\varphi|^2-|B_{\pa E}|^2\varphi^2\bigr)\, d\H \\
& \quad+8\gamma \int_{\pa E}  \int_{\pa E}G(x,y)\varphi(x)\varphi(y)d\H(x)\,d\H(y)
 +4\gamma\int_{\pa E}\partial_{\nu}v\varphi^2\,d\H\,,
\end{split}
\end{equation*}
where $\nu$ is the outer normal to $\pa E$,  $|B_{\pa E}|^2$ is the sum of the squares of the principal curvatures of $\pa E$, and $D_\tau$ is the tangential gradient.
Note that the condition $\int_{\pa E}\vphi\, d\H=0$ is related to the fact that we consider local minimizers of $J$ under a volume constraint. It is easily checked that if $E$ is a local minimizer, then $\pa^2 J(E)$ is positive semidefinite. 

Therefore, it is natural to look for sufficient  conditions for minimality based on the positivity of    $\pa^2 J(E)$.
However, we have to take into account that $J$ is translation invariant, so that in particular $J(E)=J(E+t\eta)$ for all 
$\eta\in \R^N$ and $t\in \R$. By differentiating   twice this identity with respect to $t$, we  obtain
$\pa^2J(E)[\eta\cdot \nu]=0$. This shows that there is always a finite dimensional subspace $T(\pa  E)$ of directions where the second variation degenerates. Thus, we are led to decompose $\Ht(\pa E)=\{\vphi\in H^1(\pa E):\, \int_{\pa E}\vphi\, d\H=0 \}$ as
$$
\Ht(\pa E)=T^\perp(\pa E)\otimes T(\pa E)\,,
$$
where $T(\pa E)$ is the subspace generated by the functions $\vphi=\nu_i$, $i=1,\dots, N$, and 
$$
T^\perp(\pa E)=\Bigl\{\vphi\in \Ht(\pa E):\, \int_{\pa E}\vphi \nu\, d\H=0\Bigr\}\,.
$$
Since our energy functional is invariant under translations, it is convenient to define the distance between two subsets of $\T$ modulo
translations in the following way:
\beq\label{recente1}
\alpha(E,F):=\min_x|E\triangle (x+F)|\,.
\eeq
The main result of the paper reads as follows.
\begin{theorem}\label{th:l1min}
Let $E\subset\T$ be a regular  critical set of $J$  such that
$$
\pa^2J(E)[\vphi]>0\qquad\text{for all $\vphi\in T^\perp(\pa E)\setminus\{0\}$.}
$$
Then,  there exist $\delta$, $C>0$  such that 
\beq\label{quantit.area}
J(F)\geq J(E)+C (\alpha(E,F))^2
\eeq
for all $F\subset\T$, with   $|F|=|E|$ and $\alpha(E,F)<\delta$.
 \end{theorem}
 A first application of the previous theorem deals with lamellar configurations. In Theorem~\ref{caffe} we show that if a horizontal strip $L$ is the unique solution of the  isoperimetric problem in $\T$,   then it is also the unique global minimizer of the non local functional \eqref{in.OKgeom} under the volume constraint,  provided $\gamma$ is sufficiently small. In the two-dimensional case it is known that a horizontal strip  minimizes  the perimeter in 
 $\mathbb{T}^2$ if and only if  the volume fraction parameter $m$  satisfies $|m|<1-\frac{2}{\pi}$.  Therefore, our Theorem~\ref{caffe} yields the global minimality of a single strip for small values of $\gamma$ if $|m|<1-\frac{2}{\pi}$, thus giving an alternative proof of a result already proved in \cite{ST}. Concerning the three-dimensional case, to the best of our knowledge nothing was known about the minimality of the lamellar configuration, apart from a classical result by Hadwiger (see \cite{H}), who proved that the strip  is the unique minimizer of the perimeter in $\mathbb{T}^3$ under the volume constraint  $\frac12$. In Section~\ref{appl} we improve this result by showing that the isoperimetric property still holds for strips with volume in a neighborhood of $\frac12$ (see Theorem~\ref{erice}). In turn, this implies via Theorem~\ref{caffe}  that such strips are also global minimizers of $J$ for $\gamma$ small. 

We also mention, as a simple consequence of Theorem~\ref{th:l1min},  that in any dimension and for any $\gamma>0$ lamellar configurations are local minimizers, provided that the number of strips is sufficiently large (see Proposition~\ref{digest}).

 It is important to remark that Theorem~\ref{th:l1min}, besides proving strict local minimality, contains a quantitative estimate of the 
 deviation from minimality for sets close to $E$ in $L^1$. This can be viewed as a quantitative  isoperimetric inequality for the nonlocal perimeter \eqref{in.OKgeom}, in the spirit of the recent results proved in \cite{FMP}, see also \cite{FiMP,CL}. Indeed, since our result  holds also  when
 $\gamma=0$, we cover the important case of  local minimizers of the area functional under periodicity conditions. 
 \begin{corollary}\label{cor:area}
 Let $E\subset \T$ be a regular set whose boundary has constant mean curvature and such that 
 $$
 \int_{\pa E}\bigl(|D_{\tau}\vphi|^2-|B_{\pa E}|^2\vphi^2\bigr)\, d\H>0\qquad\text{for all $\vphi\in T^\perp(\pa E)\setminus\{0\}$.}
 $$ 
 Then,  there exist $\delta$, $C>0$  such that 
$$
P_\T(F)\geq P_\T(E)+C (\alpha(E,F))^2
$$
for all $F\subset\T$, with   $|F|=|E|$ and $\alpha(E,F)<\delta$.
 \end{corollary}
Previous related investigations  were carried out by B.~White \cite{white} and K.~Grosse-Brauckmann \cite{Gr}, who proved that the strict positivity of the second variation implies local minimality with respect to small $L^\infty$-perturbations. Their results were recently extended by F.~Morgan and A.~Ros in \cite{MR}, where they show that strictly stable constant mean curvature hypersurfaces are area minimizing with respect to small $L^1$-perturbations, up to dimension $N=7$, thus giving a positive answer to a conjecture formulated in \cite{CS2}. Our corollary removes the restriction $N\leq 7$ and improves their result in a quantitative fashion.

Notice that  Corollary~\ref{cor:area} applied to the unit ball $E$ and with $\T$ replaced by $c\T$ for $c>0$ sufficiently large,  yields the quantitative isoperimetric inequality in the standard Euclidean case for bounded open sets $F$ with small asymmetry index $\alpha(E,F)$. This fact, in view of 
Lemma~5.1 in \cite{FMP}, implies the quantitative isoperimetric inequality for all sets, \emph{thus leading to an alternative proof based on the second variation.}

The Neumann counterpart to Theorem~\ref{th:l1min} is stated and proved in Section~\ref{secneu}.

%

We now briefly describe the strategy of the proof of Theorem~\ref{th:l1min}. The first step is to show that strict stability implies local minimality with respect to $W^{2,p}$-perturbations, see Theorem~\ref{th:c2min}. This is accomplished by constructing suitable volume-preserving flows connecting 
the critical set $E$ to a given close competitor $F$ and by carefully analyzing the continuity properties of the quadratic form 
$\pa^2J$ along the flow (see Theorem~\ref{prop:connect}). A technical difficulty in this analysis comes from the translation invariance, since we have to avoid the degenerate directions at all times. This issue is dealt with in Lemma~\ref{lm:italiano}, where it is shown that given any set $F$ sufficiently $W^{2,p}$-close to $E$, one can always find a translation of $F$ such that the function describing the boundary of the new set has small component in $T^\perp(\partial E)$.

The second step is to show that any $W^{2,p}$-local minimizer is in fact an $L^1$-local minimizer. This is done  by a contradiction argument: we assume that there exists
a sequence $E_h$ of sets such that $|E_h|=|E|$, and $E_h\to E$ in $L^1$, but inequality \eqref{quantit.area} fails along the sequence. Then, following an idea used in \cite{FM} for a  two dimensional problem related to epitaxial growth, we replace the sequence $E_h$ with a new sequence $F_h$ of minimizers of suitable penalized problems, tailored in such a way that  \eqref{quantit.area} still fails. 
Using regularity techniques we then show that in fact the sets $F_h$  have uniformly bounded curvatures and  converge to $E$ strongly in $W^{2,p}$, thus contradicting the 
$W^{2,p}$-local minimality of $E$. A penalization approach via regularity has been recently used also in \cite{CL} to prove the quantitative isoperimetric inequality in the Euclidean case. However, our method is quite different and seems more suited to deal with local minimizers. 

We now state a  result that links Theorem~\ref{th:l1min} with the existence of local minimizers for the
Ohta-Kawasaki energy \eqref{in.OK}. 
Fix $m\in (-1, 1)$. We say that a function $u\in H^1(\T)$ is an \emph{isolated local minimizer} for the functional $\mathcal E_\e$ with prescribed volume $m$, if $\int_\T u\, dx =m$ and there exists $\de>0$ such
that 
$$
\mathcal E_\e(u)<\mathcal E_\e(w) \qquad \text{for all $w\in H^1(\T)$ with }\int_\T w\, dx=m,\quad
0<\min_{\tau}\|u-w(\cdot+\tau)\|_{L^1(\T)}\leq \de\,.
$$
Since it is well-known that the functionals $\mathcal E_\e$ only $\Gamma$-converge in $L^1$ to the sharp interface energy
 $J$, the $L^1$-local minimality result proved in Theorem~\ref{th:l1min} allows to show:
\begin{theorem}\label{th:OK}
Let $E$ be a regular  critical set  for the functional $J$ with positive second variation and $u=\chi_E-\chi_{\T\setminus E}$. 
Then there exist  $\e_0>0$ and a family $\{u_\e\}_{\e<\e_0}$ of isolated local minimizers of $\mathcal E_\e$ with prescribed volume 
$m=\int_\T u\, dx$ such that 
$u_\e\to u$ in $L^1(\T)$ as $\e\to 0$.
\end{theorem}
An analogous result holds in the Neumann case, see Theorem~\ref{th:OKneu}. We stress that the choice of the $L^1$ topology in the minimality result stated in Theorem~\ref{th:l1min} is crucial in the proof of Theorem~\ref{th:OK}.

We conclude this introduction by observing that Theorem~\ref{th:OK} and its Neumann counterpart apply to a wealth of examples of 
strictly stable critical configurations for the sharp interface functional in \eqref{in.OKgeom}. Among these,  we mention the many droplet and spherical patterns proved to be strictly stable  in \cite{RW3} and \cite{RW4}, for some range of the parameters involved. In particular,  we can deduce that  for small values of $\e$ there are local minimizers of the diffuse energies \eqref{in.OK} which are close  to such configurations, thus solving a problem which was left open in the aforementioned papers.   

A straightforward variant of the argument used to prove Theorem~\ref{th:OK} shows also that if $\pa E$ is a periodic strictly stable constant mean curvature hypersurface, then for sufficiently small values of $\e$ and $\gamma_0$ in \eqref{in.OK} there exist
local minimizers of $\mathcal E_\e$ which are close to $E$. This seems to gives a mathematical confirmation to the findings of Thomas et al. \cite{TAHH}, 
who observed domain structures in phase-separated diblock copolymers that closely approximate 
triply periodic constant mean curvature surfaces, such as the gyroids. Indeed, strict stability for a class of triply periodic surfaces was proved in \cite{Ross}.

The paper is organized as follows: in section~\ref{stat} we give the precise mathematical formulation of the problem and we prove some preliminary results concerning the regularity of local minimizers; section~\ref{minlocsecvar} is devoted to the proof of the $W^{2,p}$-local minimality of critical configurations with positive second variation. In section~\ref{sec:locmin} we show that any $W^{2,p}$-local minimizer is in fact an $L^1$-local minimizer: this result is used to complete the proof of Theorem~\ref{th:l1min}. Section~\ref{appl} is devoted to the minimality properties of lamellar configurations. The extension to the Neumann case is contained in section~\ref{secneu}, and in the final appendix we collect a few technical results and computations. 

\section{Notation and auxiliary results}\label{stat}
In the following we shall denote by $\T$ the $N$-dimensional flat torus of unit volume, i.e., the quotient of $\R^N$  under
 the equivalence relation
 $
 x\sim y \iff x-y\in \Z^N\,.
 $
 Thus, the functional space $W^{k,p}(\T)$, $k\in \N$, $p\geq 1$, can be identified with the subspace of $W^{k,p}_{loc}(\R^N)$ of functions  that are one-periodic with respect to all coordinate directions. Similarly  
 $C^{k,\alpha}(\T)$, $\alpha\in (0,1)$ denotes the space of  one-periodic functions in  $C^{k,\alpha}(\R^N)$.
 
 We now recall the definition of a function of bounded variation in the periodic setting considered in the paper.
 We say that a function $u\in L^1(\T)$ is of bounded variation if its total variation
 $$
 |Du|(\T):=\sup\biggl\{\int_\T u\,\Div \zeta\, dx:\, \zeta\in C^{1}(\T, \R^N)\,, |\zeta|\leq 1 \biggr\}
 $$
 is finite. We denote the space of such functions by $BV(\T)$. We say that a measurable set  $E\subset \T$ is of 
 \emph{finite perimeter in $\T$} if its characteristic function $\chi_E\in BV(\T)$. The perimeter $\P(E)$ of $E$ in $\T$ is nothing but  the total variation $|D\chi_E|(\T)$. We refer to \cite{AFP} for all the main properties of sets of finite perimeter needed in the following.

 For  fixed $m\in (-1,1)$ and $\gamma\geq 0$ we consider the following nonlocal variational problem:
\begin{equation}\label{egamma}
\text{minimize}\quad \mathcal{E}(u):=\frac12|Du|(\T)+\gamma\int_\T|\nabla v|^2\, dx\,, 
\end{equation}
over all $u\in BV(\T; \{-1,1\})$, with
\begin{equation}\label{nol}
 -\Delta v=u-m\quad \text{ in }\T\,, \quad \int_\T v\, dx=0\,,\quad \text{where}\, \int_\T u\, dx=m\,;
\end{equation}
the equation is to be understood in the periodic sense.
Notice that 
\begin{equation}\label{G1}
\begin{split}
\int_\T |\nabla v|^2\, dx& =-\int_\T v\Delta v\, dx=  \int_\T v(u-m)\, dx\\
&= \int_\T vu\, dx=\int_\T\int_\T G(x,y)u(x)u(y)\, dxdy\,,
\end{split}
\end{equation}
where $G(x,y)$ is the solution of
\begin{equation}\label{G2}
-\Delta_y G(x,y)=\delta_x-1\quad\text{in $\T$,} \quad\int_\T G(x,y)\, dy=0\,.
\end{equation}
Here $\delta_x$ denotes the Dirac measure supported at $x$.

From now on, we regard $\E$ as a geometric functional defined on sets of finite perimeter. Precisely,
given $E\subset \T$ such that $|E|-|\T\setminus E|=m$, we set
\begin{equation}\label{J}
J(E):=P_\T(E)+\gamma\int_\T|\nabla v_E|^2\, dx
\end{equation}
where 
\begin{equation}\label{vE}
-\Delta v_E=u_E-m\quad \text{ in }\T\,,\quad\text{ with }\quad u_E:=\chi_E-\chi_{\T\setminus E}\,.
\end{equation}
\begin{remark}\label{rem:univ}
Notice that by standard elliptic regularity $v_E\in W^{2,p}(\T)$ for all 
$p\in [1,+\infty)$.
More precisely, given $p>1$, there exists a constant $C=C(p, N)$ such that
\beq\label{stimauniv}
\|v_E\|_{W^{2,p}(\T)}\leq C\qquad\text{for all $E\subset\T$ such that $|E|-|\T\setminus E|=m$.}
\eeq
\end{remark}

It can be shown (see \cite[Theorem 2.3]{CS}) that if $E$ is a sufficiently smooth (local) minimizer of the functional \eqref{J}, then the following 
Euler-Lagrange equation holds:
\begin{equation}\label{ELJ}
H_{\pa E}(x)+4\gamma v_E(x)=\lambda \qquad\text{for all $x\in \partial E$,}
\end{equation}
where $\lambda$ is a constant Lagrange multiplier associated to the volume constraint and $H_{\pa E}(x)$ denotes the sum of the principal curvatures of $\partial E$ at $x$; i.e., $H_{\pa E}(x)=\Div_{\tau} \nu^E$, where $\nu^E$ is the outer unit normal to $\pa E$ and
$\Div_{\tau}$ denotes the tangential divergence on $\pa E$ (see \cite[Section 7.3]{AFP}). When no confusion is possible, we shall omit the dependence of the outer unit normal on the set. 
\begin{definition}\label{def:noncianome}
 We say that $E\subset \T$ is \emph{a regular critical set} for the functional \eqref{J} if $E$ is of class $C^1$ and 
\eqref{ELJ} holds on $\partial E$ in the weak sense; i.e, 
$$
\int_{\pa E} \Div_{\tau} \zeta\, d\H=-\int_{\pa E}4\gamma v_E (\zeta\cdot \nu)\, d\H\quad \text{for all $\zeta\in C^1(\T; \R^N)$ s.t. }\int_{\pa E}\zeta\cdot\nu\, d\H=0\,. 
$$
\end{definition}

\begin{remark}\label{rem:critic}
 By Remark~\ref{rem:univ}, if $E$ is a regular critical set, then from \eqref{ELJ}, by standard regularity (see \cite[Theorem 7.57]{AFP}) we have that $E$ is of class $C^{1,\alpha}$ for all $\alpha\in (0,1)$. In turn,  Schauder estimates imply that
 $E$ is of class $C^{3,\alpha}(\T)$ for all   $\alpha\in (0,1)$.
 \end{remark}


\begin{definition}\label{def:locmin}
Recalling \eqref{recente1}, we say that a set $E\subset \T$ of finite perimeter is a {\em local minimizer} for the functional \eqref{J} if there esists $\delta>0$ such that
$$
J(F)\geq J(E)
$$
for all $F\subset\T$ with $|E|=|F|$ and $\alpha(E, F)\leq \delta$. If the inequality is strict whenever $\alpha(E,F)>0$, then we say that 
$E$ is an {\em isolated local minimizer}. We say that $E$ is a {\em regular local minimizer} if, in addition, it is a regular critical set according to Definition~\ref{def:noncianome}.
\end{definition}

We also recall the definition of $\omega$-minimizers for the area functional.

\begin{definition}\label{def:omegamin}
We say that a set of finite perimeter $E\subset \T$ is an $\omega$-minimizer for the area functional, $\omega>0$, if for any ball 
$B_r(x_0)\subset\T$ and any set of finite perimeter $F\subset \T$ such that $E\triangle F\CC B_r(x_0)$ we have
$$
\P(E)\leq \P(F)+\omega r^N\,.
$$
\end{definition}

Proposition~\ref{prop:EF} below shows that the volume constraint can be removed and replaced by a sufficiently large volume penalization term. Before proving it, we need the following lemma.

\begin{lemma}\label{lm:v}
There exists $C=C(N)>0$ such that if $E$, $F\subset\T$ are  measurable, then   
$$
\biggl|\int_{\T}|\nabla v_E|^2\, dx-\int_{\T}|\nabla v_F|^2\, dx\biggr|\leq C|E\triangle F|\,,
$$
where $v_E$ and $v_F$ are defined as in \eqref{vE}.
\end{lemma}
\begin{proof}
Note that
$$
\int_\T|\nabla v_{E}|^2\, dx-\int_\T|\nabla v_{F}|^2\, dx=\int_\T|\nabla v_{E}-\nabla v_{F}|^2\,dx +
2\int_\T\nabla v_{F}\cdot(\nabla v_{E}-\nabla v_{F})\, dx\,.
$$
Since 
$$
-\Delta(v_{E}-v_{F})=2(\chi_{E}-\chi_{F})-2(|E|-|F|)\,,
$$
 we have
 $$
 \int_\T|\nabla v_{E}-\nabla v_{F}|^2\,dx\leq c\int_\T|\chi_{E}-\chi_{F}+|F|-|E||^2\, dx\leq c|E\triangle F|\,.
 $$
 Moreover, 
 $$
 \int_\T\nabla v_{F}\cdot(\nabla v_{E}-\nabla v_{F})\, dx=2 \int_\T v_{F}(\chi_{E}-\chi_{F}+|F|-|E|)\, dx\leq
 c| E\triangle F|
 $$
 so that we may conclude that 
$$
 \biggl|\int_\T|\nabla v_{E}|^2\, dx-\int_\T|\nabla v_{F}|^2\, dx\biggr|\leq C| E\triangle F|\,.
$$

\end{proof}

\begin{proposition}\label{prop:EF}
Let $E$ be a local minimizer for the functional \eqref{J} and let $\delta>0$ be as in {\rm Definition~\ref{def:locmin}}. Then there exists $\lambda>0$ such that
$E$ solves the following penalized minimization problem:
$$
\min\Bigl\{J(F)+\lambda||F|-|E||:\, F\subset\T\,,\, \alpha(E,F)\leq\frac\delta2\Bigr\}\,.
$$
 \end{proposition}
 
 \begin{proof}
 We adapt to our situation an argument  from \cite[Section 2]{EF}. We indicate only the relevant changes.
 We set
 $$
 J_\lambda(F):=J(F)+\lambda||F|-|E||\,.
 $$
 We argue by contradiction assuming that there exist a sequence $\lambda_h\to\infty$ and a sequence $E_h$ such that
 $$
 J_{\lambda_h}(E_h)=\min\Bigl\{J_{\lambda_h}(F):\, \alpha(E,F)\leq\frac\delta2\Bigr\}\,,
 $$
 but $|E_h|\neq |E|$. Without loss of generality we may assume that  $|E_h|<|E|$ (the other case being similar) and $E_h\to \widetilde E$, with $|\widetilde E|=|E|$ and $\alpha(E, \widetilde E)\leq \frac\delta2$. Notice that the compactness of $E_h$ follows from the fact that
 $J_{\lambda_h}(E_h)\leq J(E)$ and thus the perimeters are equibounded.
 
 Arguing as in Step 1 of  \cite{EF}, given $\e>0$ we can find  $r>0$ and a point $x_0\in \T$ such that  
 $$
 | E_h\cap B_{r/2}(x_0)|<\e r^N\,, \quad | E_h\cap B_{r}(x_0)|>\frac{\omega_Nr^N}{2^{N+2}}
 $$
 for all $h$ sufficiently large. To simplify the notation we assume that $x_0=0$ and we write $B_r$ instead of $B_r(0)$. For a sequence $0<\sigma_h<1/2^N$ to be chosen, we introduce the following sequence of bilipschitz maps:
  $$
  \Phi_h(x):=
 \begin{cases}
 (1-\sigma_h(2^N-1))x & \text{if $|x|\leq \frac r2$,}\\
 x+\sigma_h\Bigl(1-\frac{r^N}{|x|^N}\Bigr)x & \text{$\frac r2\leq |x|<r$,}\\
 x & \text{$|x|\geq r$.}
 \end{cases}
 $$
 Setting $\widetilde E_h:=\Phi_h(E_h)$, we have as in Step 3 of \cite{EF}
 \beq\label{EFper}
 P_{B_r}(E_h)-P_{B_r}(\widetilde E_h)\geq -2^NNP_{B_r}(E_h)\sigma_h\,.
 \eeq
 Moreover, as in Step 4  of \cite{EF} we have
$$
 |\widetilde E_h|-|E_h|\geq \sigma_hr^N\Bigl[c\frac{\omega_N}{2^{N+2}}-\e(c+(2^N-1)N)\Bigr]
$$
 for a suitable constant $c$ depending only on the dimension $N$.
 Let us fix $\e$ so that the negative term  in the square bracket does not exceed half  the  positive one, we have that 
  \beq\label{EFvol}
   |\widetilde E_h|-|E_h|\geq \sigma_hr^NC_1\,,
  \eeq
  with $C_1>0$ depending on $N$.
  In particular from this inequality it is clear that we can choose $\sigma_h$ so that $|\widetilde E_h|=|E|$; this implies that $\sigma_h\to 0$.
 
By Lemma~\ref{lm:v} we have 
 \beq\label{Lambda-min}
 \biggl|\int_\T|\nabla v_{E_h}|^2\, dx-\int_\T|\nabla v_{\widetilde E_h}|^2\, dx\biggr|\leq C_0|\widetilde E_h\triangle E_h|\,.
 \eeq
 
 Let us now estimate $|\widetilde E_h\triangle E_h|$.
 To this aim observe that if $f\in C^1(\T)$
\begin{align}\label{come}
 \int_\T|f(\Phi_h^{-1}(x))&-f(x)|\, dx \leq \int_\T\int_0^1|\nabla f(tx+(1-t)\Phi_h^{-1}(x))||\Phi_h^{-1}(x)-x|\, dtdx\nonumber\\
 &\leq c\sigma_h
\int_0^1 \int_{B_r}|\nabla f(tx+(1-t)\Phi_h^{-1}(x))|\,dxdt\leq c\sigma_h\int_{B_r}|\nabla f(y)|\, dy\,,
\end{align}
 where the last inequality is obtained by a change of variables. By approximation we deduce
\beq\label{difsim}
  |\widetilde E_h\triangle E_h|=\int_\T|\chi_{E_h}(\Phi_h^{-1}(x))-\chi_{E_h}(x)|\, dx\leq C_3\sigma_hP_{B_r}(E_h)\,.
\eeq
 Notice that, in particular, since $\sigma_h\to 0$ for $h$ sufficiently large  we have that $\alpha(\widetilde E_h, E)\leq \delta$.
 Combining \eqref{EFper}, \eqref{EFvol}, \eqref{Lambda-min}, and \eqref{difsim} we conclude that for $h$ sufficiently large
 $$
 J_{\lambda_h}(\widetilde E_h)\leq J_{\lambda_h}(E_h)+\sigma_h\bigl[(2^NN+\gamma C_0C_3)P_{B_r}(E_h)-\lambda_hr^NC_1\bigr]
 <J_{\lambda_h}(E_h)\,,
$$
a contradiction to the minimality of $E_h$.
 \end{proof}
 As a consequence of two  previous results we recover the following regularity result which was proved first in \cite{ST}.
  \begin{theorem}\label{th:regularity}
 Let $E$ be a local minimizer for \eqref{J}. Then $E$ is an $\omega$-minimizer for the area functional. Moreover, the reduced boundary $\partial^*E$ is a $C^{3,\alpha}$ manifold for all $\alpha<1$ and  the Hausdorff dimension of the singular set satisfies
  ${\rm dim}_{\mathcal H}(\partial E\setminus \partial^*E)\leq N-8$. 
 \end{theorem}
 \begin{proof}
 We start by showing that $E$ is an $\omega$-minimizer for the area functional for a suitable $\omega>0$. To this aim 
 fix any ball $B_r\subset \T$ such that $\omega_Nr^N\leq\delta/2$, where $\delta$ is like in Definition~\ref{def:locmin}.
 Using Proposition~\ref{prop:EF}, we may find $\lambda>0$ such that $E$ minimizes $J_\lambda$ among all $F\subset\T$ with
 $\alpha(E,F)\leq\delta/2$. Therefore, if $F$ is any set of finite perimeter coinciding with $E$ outside $B_r(x_0)$, using an estimate
 similar to \eqref{Lambda-min}, we have 
 \begin{align*}
 P_{B_r}(E)-P_{B_r}(F)&=J_\lambda(E)- J_\lambda(F)+
 \gamma \int_\T\bigl(|\nabla v_F|^2-|\nabla v_E|^2\bigr)\, dx+\lambda||F|-|E||\\
 &\leq \gamma C_0|E\triangle F|+\lambda||F|-|E||\leq (\gamma C_0+\lambda)\omega_Nr^N\,.
 \end{align*}
 This shows that $E$ is an $\omega$-minimizer with $\omega:=(\gamma C_0+\lambda)\omega_N$. 
 By classical regularity results (see \cite[Theorem 1]{Ta}), it follows that  $\partial^*E$ is a $C^{1,\frac12}$-manifold and
  ${\rm dim}_{\mathcal H}(\partial E\setminus \partial^*E)\leq N-8$.
  The  $C^{3,\alpha}$ regularity then follows from Remark~\ref{rem:critic}.
 \end{proof}
 \begin{remark}\label{rm:noncianome}
 Observe that the $C^{3,\alpha}$  regularity follows only from the equation. Hence, in view of Remark~\ref{rem:critic} it holds for regular critical sets.
 \end{remark}
\section{Second variation and $W^{2,p}$-local minimality}\label{minlocsecvar}

Let $E\subset\T$ be of class $C^2$ and $X:\T\to\T$ a $C^2$-vector field, and consider the associated flow 
$\Phi:\T\times(-1,1)\to\T$ defined by
$\displaystyle\frac{\pa \Phi}{\pa t}=X(\Phi)$, $\Phi(x, 0)=x.$
 We define the \emph{second variation of $J$ at $E$} with respect
 to the   flow $\Phi$ to be the value 
 $$
 \frac{d^2 }{dt^2}J(E_t)_{\bigl|t=0}\,,
 $$
 where $E_t:=\Phi(\cdot, t)(E)$.
 
 Throughout the section, when no confusion is possible, we shall omit the indication of $E$, writing $v$ instead of $v_E$, $\nu$ instead of $\nu^E$, and denoting by $d$ the signed distance from the boundary of $E$.
 
Before stating the representation formula for the second variation, we fix some notation. Given a vector $X$, its tangential part on $\pa E$ is defined as $X_\tau:=X-(X\cdot \nu)\nu$. In particular, we will denote by $D_\tau$ the tangential gradient operator given by 
$D_\tau\vphi:=(D\vphi)_\tau$.   We  also recall that  the second fundamental form
$B_{\pa E}$ of $\pa E$ is given by $D_\tau\nu$ and that  the square $|B_{\pa E}|^2$  of its  Euclidean norm coincides with
the  the sum of the squares of the principal curvatures of $\pa E$. 

\begin{theorem}\label{th:J''}
If $E$, $X$, and $\Phi$ are as above, we have
\begin{equation}\label{eq:J''}
\begin{split}
\frac{d^2 }{dt^2}J(E_t)_{\bigl|t=0}& = \int_{\pa E}\Bigl(|D_\tau(X\cdot \nu)|^2-|B_{\pa E}|^2(X\cdot \nu)^2\Bigr)\, d\H \\
& \quad+8\gamma \int_{\pa E}  \int_{\pa E}G(x,y)\bigl((X\cdot \nu)(x)\bigr)\bigl((X\cdot \nu)(y)\bigr)d\H(x)\,d\H(y)\\
& \quad+4\gamma\int_{\pa E}\partial_{\nu}v(X\cdot \nu)^2\,d\H
-\int_{\pa E}(4\gamma v+H_{\pa E})\Div_{\tau}\bigl(X_\tau(X\cdot\nu)\bigr)\, d\H\\
& \quad+\int_{\pa E}(4\gamma v+H_{\pa E})(\Div X)(X\cdot\nu)\, d\H\,.
\end{split}
\end{equation}
\end{theorem}
The proof of the theorem is given in the Appendix. 
\begin{remark}\label{NonSiamoCS}
In the case of a critical set $E$ the computation of the second variation was carried out in \cite{CS}. The novelty here is that we deal 
with a general regular set. This explains the presence of the last two terms in the formula.
\end{remark}
\begin{remark}\label{nn}
Notice that if $E$ is also critical, from \eqref{ELJ} it follows that 
$$
\int_{\pa E}(4\gamma v+H_{\pa E})\Div_\tau\bigl(X_\tau(X\cdot\nu)\bigr)\, d\H=0\,.
$$ 
Moreover, if in addition  
\beq\label{recente2}
|\Phi(\cdot,t)(E)|=|E| \qquad\text{for all $t\in [0,1]$},
\eeq
 then  it can be shown (see \cite[equation (2.30)]{CS}) that
$$
0=\frac{d^2 }{dt^2}|E_t|_{\bigl|t=0}=\int_{\pa E}(\Div X)(X\cdot\nu)\, d\H\,.
$$
Hence, again from \eqref{ELJ}, we have
\begin{equation}\label{nn1}
\begin{split}
\frac{d^2 }{dt^2}J(E_t)_{\bigl|t=0}& = \int_{\pa E}\Bigl(|D_\tau(X\cdot \nu)|^2-|B_{\pa E}|^2(X\cdot \nu)^2\Bigr)\, d\H \\
&\quad +8\gamma \int_{\pa E}  \int_{\pa E}G(x,y)\bigl((X\cdot \nu)(x)\bigr)\bigl((X\cdot \nu)(y)\bigr)d\H(x)\,d\H(y)\\
&\quad  +4\gamma\int_{\pa E}\partial_{\nu}v(X\cdot \nu)^2\,d\H\,.
\end{split}
\end{equation}
Note that this formula coincides exactly with the one given in \cite[equation (2.20)]{CS}, where it was obtained using a particular family of asymptotically volume preserving diffeomorphisms.
\end{remark}
The previous remark motivates the following definition. Given any sufficiently smooth open set $E\subset\T$ we denote by $\Ht(\pa E)$ the set of all functions $\varphi\in H^1(\pa E)$ such that $\int_{\pa E}\varphi\,d\H=0$, endowed with the norm
$\|\nabla \vphi\|_{L^2(\pa E)}$. To  $E$ we then associate the quadratic form $\pa^2J(E):\Ht(\pa E)\to\R$ defined as
\begin{equation}\label{pa2J}
\begin{split}
\pa^2J(E)[\varphi]& = \int_{\pa E}\bigl(|D_\tau\varphi|^2-|B_{\pa E}|^2\varphi^2\bigr)\, d\H \\
& \quad+8\gamma \int_{\pa E}  \int_{\pa E}G(x,y)\varphi(x)\varphi(y)d\H(x)\,d\H(y)
 +4\gamma\int_{\pa E}\partial_{\nu}v\varphi^2\,d\H\,.
\end{split}
\end{equation}
If $E$ is a regular critical set and the flow $\Phi$ satisfies \eqref{recente2}, then 
$$
\frac{d|E_t|}{dt}\Bigr|_{t=0}=\int_{\pa E}X\cdot \nu\, d\H=0\,.
$$
Hence, $\pa^2J(E)[X\cdot\nu]$ coincides with the second variation of $J$ at $E$ with respect to $\Phi$.

Notice that, setting $\mu:=\varphi\H\lfloor\pa E$, the nonlocal term
$$
\int_{\pa E} \int_{\pa E}G(x,y)\varphi(x)\varphi(y)d\H(x)\,d\H(y)
$$
can be rewritten as
\beq\label{nonloc}
 \int_{\T} \int_{\T}G(x,y)d\mu(x)\,d\mu(y)=\int_\T|\nabla z|^2\,dx\,,
\eeq
where $z\in H^1(\T)$ is a weak solution to the equation
$$
-\Delta z=\mu\qquad\text{in}\,\,\T\,.
$$
Thus the nonlocal term \eqref{nonloc}  is equivalent to the square of the $H^{-1}$-norm of the measure $\mu$.
\par
As a consequence of Remark~\ref{nn}  we have the following Corollary.
\begin{corollary}\label{co:J''}
Let $E$ be a regular local minimizer of $J$ according to {\rm Definition~\ref{def:locmin}}.Then
$$
\pa^2J(E)[\varphi]\geq0\qquad\qquad\text{for all $\varphi\in\Ht(\pa E)$.}
$$
\end{corollary}
\begin{proof}
Let $\varphi\in C^\infty\cap \Ht(\pa E)$. We set   $X:=\nabla u$ where $u$ solves
$$
\begin{cases}
\Delta u=0 & \text{in $\T\setminus \pa E$,}\\
\pa_\nu u=\vphi & \text{on $\pa E$,}
\end{cases}
$$
where $\nu$ is the outer normal to $\pa E$. Note that $\Div X=0$ and, by elliptic regularity and recalling Remark~\ref{rem:critic},  $X\cdot \nu$ is of class $C^{2,\alpha}$ separately 
in $\bar E\cap \mathcal{N}(\pa E)$ and $\overline{\T\setminus E}\cap \mathcal{N}(\pa E)$ and globally Lipschitz continuous in $ \mathcal{N}(\pa E)$
 for a suitable tubular neigborhood $\mathcal{N}(\pa E)$ of $\pa E$. Here by $\nu$ we denote a $C^{2,\alpha}$ extension of the outer unit normal field $\nu^E$ from $\pa E$ to $\mathcal{N}(\pa E)$. With a slight abuse of notation we still denote by $D_{\tau}$ the 
 extension of the tangential gradient on $\pa E$ given by $D_{\tau}:=D-\nu\pa_\nu$ in $\mathcal{N}(\pa E)$.  Observe that $D_\tau(X\cdot\nu)$ is continuous in $\mathcal{N}(\pa E)$.
We now set 
$$
X_\e(x):=\int_\T\rho_\e(z)X(x+z)\, dz\,,
$$
where $\rho_\e$ is the standard mollifier. Notice that $\Div X_\e=0$. Hence, the associated flow is volume preserving and by the local minimality together with \eqref{nn1} we have 
$\pa^2J(E)[\vphi_\e]\geq 0$, where $\vphi_\e:=X_\e\cdot \nu$. 
We claim that $\vphi_\e\to \vphi$ in $C^1(\pa E)$. Indeed, observing that we can write
$$
(X_\e\cdot \nu)(x)=(X\cdot \nu)_\e(x)-
\int_\T\rho_\e(z)X(x+z)\cdot[\nu(x+z)-\nu(x)]\, dz=:
(X\cdot \nu)_\e(x)-R_\e(x)
$$
and recalling that $X\cdot\nu$ is continuous in $\mathcal N(\pa E)$, one easily gets that $X_\e\cdot \nu\to X\cdot \nu$ uniformly in  $\mathcal N(\pa E)$. In
particular, $\vphi_\e\to \vphi$ uniformly on $\pa E$. To show that 
$D_\tau (X_\e\cdot \nu)\to D_\tau (X\cdot \nu)$ uniformly in  $\mathcal N(\pa E)$, it is enough to check (by a lengthy but straightforward computation)   that $\nabla R_\e\to 0$ 
 uniformly in $\mathcal N(\pa E)$. Hence, the claim follows recalling the continuity of  $D_\tau (X\cdot \nu)$. It is now easy to check that the claim implies that $\pa ^2J(E)[\vphi]=\lim_\e\pa ^2J(E)[\vphi_\e]\geq 0$. If  now $\vphi$ is any function in $\Ht(\pa E)$, we construct  a
sequence $\vphi_n$ of functions in $C^{\infty}(\pa E)\cap \Ht(\pa E)$ such that $\vphi_n\to \vphi$ in $H^1(\pa E)$. Then the conclusion follows
by observing that all the terms appearing in the expression of $\pa^2J( E)$ are continuous with respect to the $H^1$-convergence.
\end{proof}

We now switch to the search for a sufficient condition for local minimality. Observe that if $E\subset\T$ is of class $C^2$ and $\Phi(x,t)=x+t\eta e_i$ for some $\eta\in\R$ and some element $e_i$ of the canonical basis in $\R^N$, we clearly have $J\bigl(\Phi(\cdot,t)(E)\bigr)= J(E)$, by the translation invariance of $J$. Hence,
$$
\frac{d^2 }{dt^2}J(E_t)_{\bigl|t=0}=\pa^2J(E)[\eta\nu_i]=0\,.
$$
In view of this it is convenient to introduce the  subspace $T(\pa E)\subset \Ht(\pa E)$ generated by the functions
$\nu_i$, $i=1,\dots, N$. Note that we can then write 
\beq\label{tperp}
\Ht(\pa E)=T^\perp(\pa E)\otimes T(\pa E)\,,
\eeq
where 
$$
T^\perp(\pa E):=\Bigl\{\vphi\in\Ht(\pa E):\, \int_{\pa E}\vphi\nu_i\, d\H=0\,, i=1,\dots, N\Bigr \}
$$
is the orthogonal set, in the $L^2$-sense, to the space of infinitesimal translations $T(\pa E)$.

We observe that there exists an orthonormal frame $\{\e_1, \dots, \e_N\}$ such that 
\beq\label{giro}
\int_{\pa E}(\nu\cdot \e_i)(\nu\cdot \e_j)\, d\H=0 \qquad\text{for all $i\neq j$.}
\eeq
The existence of such orthonormal frame can be proved by observing that, denoting by $A$ the matrix with coefficients 
$a_{ij}:=\int_{\pa E}\nu_i\nu_j\, d\H$, we have for every $O\in SO(N)$
$$
\int_{\pa E}(O\nu)_i(O\nu)_j\, d\H=\bigl(OAO^{-1}\bigr)_{ij}\,.
$$
Choose $O$ so that $OAO^{-1}$ is diagonal and set $\e_i=O^{-1}e_i$. In view of this remark, the functions
$\nu\cdot\e_i$ are orthogonal and generate   $T(\pa E)$. Notice however that the dimension of $T(\pa E)$ can be strictly smaller than $N$, since it may happen that $\nu\cdot\e_i=0$ for some $i$, as in the case when $E$ is translation invariant  along some direction. Therefore, given $\vphi\in \Ht(\pa E)$, its projection on $T^\perp(\pa E)$ is 
\beq\label{projtperp}
\pi_{T^\perp(\pa E)}(\vphi)=\vphi-\sum\left(\int_{\pa E}\vphi \nu\cdot \e_i\, d\H\right)\frac{\nu\cdot \e_i}{\|\nu\cdot \e_i\|^2_2}\,,
\eeq
where it is understood that the sum runs over all indices $i$ such that $\|\nu\cdot \e_i\|_{L^2(\pa E)}\neq 0$.
\begin{definition}\label{def:psv}
In the following we say that the functional $J$ has \emph{positive second variation} at the critical set $E$ if 
$$
\pa^2J(E)[\vphi]>0 \qquad \text{for all $\vphi\in \Ht(\pa E)\setminus T(\pa E)$}
$$
or, equivalently, for all $\vphi\in T^\perp(\pa E)\setminus\{0\}$.
\end{definition}
\begin{lemma}\label{coerc}
Assume that  $J$ has positive second variation at the critical set $E$. Then 
\beq\label{c0}
m_0:=\inf\Bigl\{\pa^2J(E)[\vphi]:\, \vphi\in T^\perp(\pa E)\,, \|\vphi\|_{H^1}=1\Bigr\}>0\,,
\eeq
and 
$$
\pa^2J(E)[\vphi]\geq m_0 \|\vphi\|^2_{H^1} \qquad \text{for all $\vphi\in T^\perp(\pa E)$.}
$$
\end{lemma}
\begin{proof}
Let $\vphi_h$ be a minimizing sequence for the infimum in \eqref{c0} and assume that $\vphi_h\wto \vphi_0\in T^\perp(\pa E)$ weakly in $H^1(\pa E)$. If $\vphi_0\neq 0$, by \eqref{pa2J} it follows that 
$$
m_0=\lim_h \pa^2 J(E)[\vphi_h]\geq \pa^2 J(E)[\vphi_0]>0\,.
$$
 If $\vphi_0= 0$, then
$$
m_0=\lim_h\pa^2 J(E)[\vphi_h]=\lim_h\int_{\pa E}|D_\tau \vphi_h|^2 \, d\H=1\,.
$$ 
\end{proof}
We now show how to construct a flow satisfying \eqref{recente2} connecting any two sufficiently regular and close sets in $\T$. If $E\subset\T$ is at least of class $C^2$, we denote by $\mathcal N_{r}(\pa E)$ the tubular neighborhood of $\pa E$ of thickness $2r$.
We shall always assume $r$ to be so small that the signed distance $d$ from $\pa E$ and the projection $\pi$ on $\pa E$ are well defined and regular on
$\mathcal N_{r}(\pa E)$; when $r$ is irrelevant, we shall omit it.
\begin{theorem}\label{prop:connect}
Let $E\subset\T$ be a set of class $C^3$ and let $p>N-1$. For all  $\e>0$ there exist a  tubular neighborhood $\mathcal N_{r}(\pa E)$ and   two positive constants $\delta$, $C$ with the following properties: If $\psi\in C^{2}(\pa  E)$ and $\|\psi\|_{W^{2,p}(\pa E)}\leq \de$ then
there exists a field $X\in C^{2}$ with $\Div X=0$  in 
$\mathcal N_{r}(\pa E)$ such that
\begin{equation}\label{connect0}
\|X-\psi\nu\|_{L^2(\pa E)}\leq \e\|\psi\|_{L^2(\pa E)}\,.
\end{equation}
Moreover,  the associated flow
\beq\label{campo-1}\Phi(x,0)=x\;,\qquad \frac{\pa\Phi}{\pa t}=X(\Phi)\eeq
satisfies $\Phi(\pa E,1)=\{x+\psi(x)\nu:\, \in\pa E\}$, and for every $t\in[0,1]$
\begin{equation}\label{campo0}
 \|\Phi(\cdot, t)-{\rm Id}\|_{W^{2,p}(\pa E)}\le C\|\psi\|_{W^{2,p}(\pa E)}\,,
\end{equation}
where ${\rm Id}$ denotes the identity map.
If in addition $E_1$ has the same volume as $E$, then for every $t$ we have $|E_t|=|E|$ and 
$$
\int_{\pa E_t}X\cdot \nu^{ E_t}\, d\H=0\,.
$$
\end{theorem}

\begin{proof} For $\sigma>0$ set $d_\sigma:=\rho_\sigma*d$, where $\rho_\sigma$ is the standard mollifier. Since $E$ is of class $C^3$ there exist a neighborhood 
$\mathcal N_{r}(\pa E)$ and $\sigma_\e$ such that if $0<\sigma<\sigma_\e$ 
\beq\label{connectone}
\|d_\sigma-d\|_{C^3(\mathcal N_{r}(\pa E))}\leq \e\,.
\eeq
For such $\sigma$ let $\Psi$ be the flow associated to $\nabla d_{\sigma}$, i.e.
$$
\Psi(x,0)=x\;,\qquad \frac{\pa\Psi}{\pa t}=\nabla d_{\sigma}(\Psi)\,.
$$
Then, there exists $t_0>0$ such that $\Psi_{|\pa E\times(-t_0,t_0)}$ is a 
$C^\infty$-diffeomorphism onto some neighborhood $U$ of $\pa E$. 
We start by constructing a $C^\infty$ vector field $\tilde X:U\to\R^N$  such that 
$$
\Div \tilde X=0\quad\text{in $U$},\qquad\quad \tilde X=\nabla d_\sigma\quad\text{on $\pa E$}.
$$
To this aim, for every $y\in U$ we set
\beq\label{connect1.5}
\zeta(y)=\zeta (\Psi(x,t)):=\exp\Bigl(-\int_0^t\Delta d_\sigma(\Psi(x,s))\, ds\Bigr)\,. 
\eeq
By construction we have that $\Div(\zeta\nabla d_\sigma)=0$ in $U$.
We define $\tilde X$ to be any $C^{\infty}$-vector field which coincides with $\zeta\nabla d_\sigma$ on $U$, and denote by $\tilde \Phi$ the associated flow. Note that
$\tilde \Phi$ and $\Psi$ have the same trajectories. 
Let us consider the two functions $\pi_\sigma: U\to \pa E$, $t_\sigma: U\to \R$ implicitly defined by
$$
\tilde \Phi (\pi_\sigma(y), t_\sigma(y))=y\,.
$$
If $t$ is small, for all $x\in \pa E$ we have $t_\sigma(\tilde \Phi(x, t))=t$. Hence, 
$\nabla t_\sigma(\tilde \Phi(x,t))\cdot \frac{\pa}{\pa t}\tilde\Phi(x,t)=1$ and in particular
$\nabla t_\sigma\cdot \nabla d_\sigma=1$ on $\pa E$. Therefore, since $t_\sigma=0$ on $\pa E$, we have
$$
\nabla t_\sigma=\frac{\nabla d}{\nabla d\cdot \nabla d_\sigma}\qquad\text{on $\pa E$.}
$$
Therefore, for $\sigma<\sigma_\e$ sufficiently small $\|\nabla t_\sigma-\nabla d_\sigma\|_{L^{\infty}(\pa E)}\leq \e$. Thus, taking $r$ smaller if needed, we may assume that $\mathcal N_{r}(\pa E)\subset U$ and for all $y\in \mathcal N_{r}(\pa E)$
$$
|t_\sigma(y)-d_\sigma(y)|\leq 2\e d(y)\,.
$$
In other words, there exists a function $a_\sigma\in C^3(\mathcal N_{r}(\pa E))$, with
$\|a_\sigma\|_{L^{\infty}(\mathcal N_{r}(\pa E))}\leq 2\e$ such that
\beq\label{connecttwo}
t_\sigma(y)=d(y)(1+a_\sigma(y))\,.
\eeq
Let us now take $\psi\in C^2(\pa E)$. If
  $\|\psi\|_{L^\infty(\pa E)}$ is small, we set 
  $$
  S(x):=\pi_{\sigma}(x+\psi(x)\nu(x))
  $$
  for $x \in \pa E$. Since $E$ is of class $C^3$ we have that $S$ is of class $C^2$. Moreover, 
  $$
  D_\tau S(x)=(D_\tau\pi_\sigma)(x+\psi(x)\nu
(x))+R(x)\,,
  $$
  where $|R(x)|\leq C\|\psi\|_{C^1(\pa E)}$. Therefore, since $\pi_\sigma(x)=x$ on $\pa E$, we deduce that $S$ is a $C^2$-diffeomorphism, provided that  $\|\psi\|_{C^1(\pa E)}$ is small. Moreover, it is easily checked if $\|D_{\tau\tau}\psi\|_{L^p(\pa E)}\leq 1$, then 
  \beq\label{connect3}
  \|S^{-1}\|_{W^{2,p}(\pa E)}\leq C
\eeq
for some positive constant $C$ independent of $\psi$. Note also that
\beq\label{connect4}
|S^{-1}(x)-x|= |S^{-1}(x)-S^{-1}(\pi_{\sigma}(x+\psi(x)\nu
(x)))|
\leq C |x-\pi_{\sigma}(x+\psi(x)\nu
(x))|\leq C|\psi(x)|\,.
\eeq
Now for $y\in \mathcal N_{r}(\pa E)$ we set 
\beq\label{connect5}
G(y):=(S^{-1}\circ\pi_\sigma)(y)+\nu
((S^{-1}\circ\pi_\sigma)(y))\psi((S^{-1}\circ\pi_\sigma)(y))\,.
\eeq
Thus, $G(y)$ is the unique point of the trajectory of $\tilde \Phi$ passing through $y$ that belongs to the graph of $\psi$.  Finally, we may define
\beq\label{connect6}
X(y):=t_\sigma(G(y))\tilde X(y)
\eeq
for $y\in \mathcal N_{r}(\pa E)$. Note that
 $X\in C^2(\mathcal N_{r}(\pa E); \T)$. We shall still denote by $X$ any $C^2$-extension of the vector field to $\T$.

Since $t_\sigma\circ G$ is constant along the trajectories of $\tilde \Phi$, we have
$\Div X=0$ in $\mathcal N_{r}(\pa E)$. Let us denote by $\Phi$ the flow associated to $X$. Since $t_\sigma(G(x))$ is the time needed to go from $x$ to $G(x)$
along the trajectory of $\tilde \Phi$, we have $\Phi(x,1)=G(x)$.
Thus, we may conclude that $\Phi(\pa E, 1)$ is the graph of $\psi$. 
Note that from \eqref{connecttwo} and \eqref{connect5}
\beq\label{connect7}
X(y)=\psi((S^{-1}\circ\pi_\sigma)(y))(1+a_\sigma(G(y)))\zeta(y)\nabla d_\sigma(y)\,.
\eeq
Thus, from \eqref{connect3} we have
\beq\label{campo3}
\|X\|_{W^{2,p}({\mathcal N}_{r}(\pa E))}\leq C \|\psi\|_{W^{2,p}(\pa E)}
\eeq
for a constant $C>0$ independent of $\psi$.

We now show \eqref{connect0}. From \eqref{connect7}, \eqref{connect3}, and \eqref{connect4}, we have for every $x\in \pa E$
\begin{align*}
|X(x)&-\psi(x)\nu
(x)|=|\psi((S^{-1}\circ\pi_\sigma)(x))(1+a_\sigma(G(x)))\zeta(x)\nabla d_\sigma(x)-\psi(x)\nabla d(x)|\\
&\leq |\psi(S^{-1}(x))| |(1+a_\sigma(G(x)))\zeta(x)\nabla d_\sigma(x)-\nabla d(x)|
+|(\psi(S^{-1}(x))-\psi(x))\nabla d(x)|\\
&\leq C\e |\psi(S^{-1}(x))|+|\psi(S^{-1}(x))-\psi(x)|\\
&\leq  C\e |\psi(S^{-1}(x))|+\|\psi\|_{C^{1}(\pa E)}|S^{-1}(x)-x|\\
&\leq C\e(|\psi(S^{-1}(x))|+|\psi(x)|)
 \end{align*}
provided that $\|\psi\|_{C^{1}(\pa E)}$ is small. Hence, \eqref{connect0} follows.

To establish  \eqref{campo0}, observe that 
 the closeness of $\Phi$ to ${\rm Id}$ in $L^\infty$ follows from \eqref{campo-1} and \eqref{campo3}. 
 By differentiating  \eqref{campo-1} and solving the resulting equation,  and since $p>N-1$,  one easily gets
 $$
 \|\nabla_x\Phi-I\|_{L^{\infty}({\mathcal N}_{\e_0}(\pa E))}\leq C(\e_0)\|\nabla X\|_{L^\infty({\mathcal N}_{\e_0}(\pa E))}\leq C(\e_0)\|\psi\|_{W^{2,p}(\pa E)}\leq C(\e_0)\e\,.
 $$
 In particular, if $\e$ is small enough the $(N-1)$-dimensional Jacobian of $\Phi(\cdot, t)$ on $\pa E$ is uniformly close to $1$. 
 Using this information and by differentiating again \eqref{campo-1}, we have
 $$
 \|\nabla^2_x\Phi(\cdot, t)\|_{L^p(\pa E)}\leq C(\e_0)\|\nabla^2X\|_{L^p({\mathcal N}_{\e_0}(\pa E))}\,,
 $$ 
 whence \eqref{campo0} follows.

Assume now that $|E_1|=|E|$ and recall that by \cite[equation (2.30)]{CS}
$$
\frac{d^2}{dt^2}|E_t|=\int_{\pa E_t}(\Div X)(X\cdot\nu^{ E_t})=0\,\qquad\text{for all $t\in [0,1]$.}
$$
Hence the function $t\mapsto |E_t|$ is affine in $[0,1]$ and since $|E_0|=|E|=|E_1|$ it is constant. Therefore
$$
0=\frac{d}{dt}|E_t|=\int_{E_t}\Div X\, dx=\int_{\pa E_t}X\cdot\nu^{ E_t}\, d\H\, \qquad\text{for all $t\in [0,1]$.}
$$
This concludes the proof of the theorem.
\end{proof}

Before proving the main result of the section we need the following key lemma, which shows that any set $F$ sufficiently close to $E$ can be translated in such a way that the resulting set $\widetilde F$ satisifies
$\pa \widetilde F=\{x+\psi(x)\nu(x):\,x\in\partial E\}$, with $\psi$ having a suitably small projection on $T(\pa E)$.
\begin{lemma}\label{lm:italiano} Let $E\subset\T$ be of class $C^3$ and let  $p>N-1$. For any $\delta>0$ there exist  $\eta_0$, $C>0$  such that if $F\subset \T$ satisfies $\partial F=\{x+\psi(x)\nu(x):\,x\in\partial E\}$ for some $\psi\in C^{2}(\pa E)$ with $\|\psi\|_{W^{2,p}(\partial E)}\leq\eta_0$, then there exist $\sigma\in\R^N$ and $\varphi\in W^{2,p}(\partial E)$  with the properties that  
$$
|\sigma|\leq C\|\psi\|_{W^{2,p}(\partial E)},\qquad \|\varphi\|_{W^{2,p}(\partial E)}\leq C\|\psi\|_{W^{2,p}(\partial E)}
$$
and 
$$
\partial F-\sigma=\{x+\varphi(x)\nu(x):\,x\in\partial E\},\qquad \Bigl|\int_{\partial E}\varphi\nu\,d{\cal H}^{N-1}\Bigr|\leq\delta\|\varphi\|_{L^2(\partial E)}\,.
$$ 
\end{lemma}
\noindent
\begin{proof} 
In the following $\nu$  stands for $\nabla d$, where $d$ is the signed distance from
 $\pa E$. Throughout the proof the various constants will be independent of $\psi$.
Set 
$$
\eta:=\|\psi\|_{W^{2,p}(\partial E)}+\|\psi\|_{L^2(\partial E)}\,.
$$
We recall  that there exists an orthonormal frame $\{\e_1, \dots, \e_N\}$ satisfying \eqref{giro}.
   
Let $I$ be the set of all $i\in \{1,\dots, N\}$ such that 
$\|\nu\cdot\e_i\|_{L^2(\pa E)}>0$. We define $\sigma=\sum_{i=1}^N\sigma_i\e_i$, where 
\begin{equation}\label{uno}
\sigma_i:=\frac{1}{\|\nu\cdot\e_i\|^2_{L^2(\pa E)}}\int_{\pa E}\psi(x)(\nu(x)\cdot \e_i)\,d\H\quad\text{if $i\in I$}\,,\qquad
\sigma_i=0\quad\text{otherwise.}
\end{equation}
Note that
\begin{equation}\label{due}
|\sigma|\leq C_1\|\psi\|_{L^2(\pa E)}\,.
\end{equation}

\noindent{\it Step 1.}  Let $T_\psi:\pa E\mapsto\pa E$ be the map
$$
T_\psi(x):=\pi(x+\psi(x)\nu(x)-\sigma)\,.
$$ 
It is easily checked that there exists $\e_0>0$ such that if 
\beq\label{transone}
\|\psi\|_{W^{2,p}(\pa E)}+|\sigma|\leq \e_0\leq 1\,,
\eeq
then $T_{\psi}$ is a diffeomorphism of class $C^2$.  Moreover, 
\beq\label{transtwo}
\bigl\|J_{N-1}d^{\pa E}T_\psi-1\bigr\|_{L^\infty(\pa E)}\leq C\|\psi\|_{C^1(\pa E)}
\eeq
and
\beq\label{trans3}
\|T_\psi^{-1}\|_{W^{2,p}(\pa E)}\leq C(\|\psi\|_{W^{2,p}(\pa E)}+|\sigma|)\,.
\eeq
Therefore, setting $\widehat F:=F-\sigma$, we have 
$$
\pa\widehat F=\{x+\varphi(x)\nu(x):\,x\in\partial E\}
$$
for some function $\varphi$, which is linked to $\psi$ by the following relation: for all $x\in\pa E$   
$$
x+\psi(x)\nu(x)-\sigma=y+\varphi(y)\nu(y)\,,
$$
where $y=T_\psi(x)$ and $\vphi(y)=d(T_\psi(x))$.
Thus, using \eqref{trans3}
 \begin{equation}\label{quattro}
\|\varphi\|_{W^{2,p}(\pa E)}\leq C_2\bigl(\|\psi\|_{W^{2,p}(\pa E)}+|\sigma|\bigr)
\end{equation}
for some $C_2\geq1$.
We now estimate
\begin{eqnarray}
\int_{\pa E}\varphi(y)\nu(y)\,d\H(y)&=&\int_{\pa E}\varphi(T_{\psi}(x))\nu(T_{\psi}(x))J_{N-1}d^{\pa E}T_\psi(x)\,d\H(x)\label{cinque}\\
&=& \int_{\pa E}\varphi(T_{\psi}(x))\nu(T_{\psi}(x))\,d\H(x)+R_1,\,\nonumber
\end{eqnarray}
where
\begin{equation}\label{sei}
|R_1|=\biggl|\int_{\pa E}\varphi(T_{\psi}(x))\nu(T_{\psi}(x))\bigl[J_{N-1}d^{\pa E}T_\psi(x)-1\bigr]\,d\H(x)\biggr|\leq C_3\|\psi\|_{C^1(\pa E)}\|\varphi\|_{L^2(\pa E)}\,.
\end{equation}
On the other hand 
\begin{eqnarray}\label{sette}
\int_{\pa E}\varphi(T_{\psi}(x))\nu(T_{\psi}(x))\,d\H&=&\int_{\pa E}\bigl[x+\psi(x)\nu(x)-\sigma-T_{\psi}(x)\bigr]\,d\H \\
&=&\int_{\pa E}\bigl[x+\psi(x)\nu(x)-\sigma-\pi(x+\psi(x)\nu(x)-\sigma)\bigr]\,d\H \nonumber\\
&=&\int_{\pa E}\bigl\{\psi(x)\nu(x)-\s+\bigl[\pi(x)-\pi(x+\psi(x)\nu(x)-\sigma)\bigr]\bigr\}\,d\H \nonumber\\
&=&\int_{\pa E}(\psi(x)\nu(x)-\s)\,d\H+R_2\,, \nonumber
\end{eqnarray}
where
\begin{eqnarray}\label{otto}
R_2&=&\int_{\pa E}\bigl[\pi(x)-\pi(x+\psi(x)\nu(x)-\sigma)\bigr]\,d\H\\
&=&-\int_{\pa E}d\H\int_0^1\nabla\pi\bigl(x+t(\psi(x)\nu(x)-\sigma)\bigr)(\psi(x)\nu(x)-\sigma)\,dt \nonumber \\
&=&-\int_{\pa E}\nabla\pi(x)(\psi(x)\nu(x)-\sigma)\,d\H+R_3\,. \nonumber
\end{eqnarray}
In turn, recalling \eqref{due}
\begin{equation}\label{nove}
|R_3|\leq\int_{\pa E}d\H\int_0^1\bigl|\nabla\pi\bigl(x+t(\psi(x)\nu(x)-\sigma)\bigr)-\nabla\pi(x)\bigr||\psi(x)\nu(x)-\sigma|\,dt \leq C_4\|\psi\|^2_{L^2(\pa E)}\,.
\end{equation}
If $x$ is sufficiently  close to $\pa E$, then $\pi(x)=x-d(x)\nu(x)=x-d(x)\nabla d(x)$ and  
$$
\frac{\pa\pi_i}{\pa x_j}(x)=\delta_{ij}-\frac{\pa d}{\pa x_i}(x)\frac{\pa d}{\pa x_j}(x)-d(x)\frac{\pa d^2}{\pa x_i\pa x_j}(x)
$$
and, thus, for all  $x\in\pa E$
$$
\frac{\pa\pi_i}{\pa x_j}(x)=\delta_{ij}-\nu_i(x)\nu_j(x)\,.
$$
From this identity and  \eqref{cinque}, \eqref{sette},  \eqref{otto} we get
$$
\int_{\pa E}\varphi(y)\nu(y)\,d\H(y)=\int_{\pa E}\bigl[\psi(x)\nu(x)-\bigl(\sigma\cdot\nu(x)\bigr)\nu(x)\bigr]\,d\H(x)+R_1+R_3\,.
$$
But the integral at the right-hand side vanishes  by \eqref{uno} and \eqref{giro}.  Therefore, \eqref{sei} and  \eqref{nove} imply
\begin{eqnarray}\label{dieci}
\Bigl|\int_{\pa E}\varphi(y)\nu(y)\,d\H(y)\Bigr|&\leq& C_3\|\psi\|_{C^1(\pa E)}\|\varphi\|_{L^2(\pa E)}+C_4\|\psi\|^2_{L^2(\pa E)} \\ \nonumber
&\leq& C\|\psi\|_{C^1(\pa E)}\bigl(\|\varphi\|_{L^2(\pa E)}+\|\psi\|_{L^2(\pa E)}\bigr) \\ \nonumber
&\leq& C_5 \|\psi\|_{W^{2,p}(\pa E)}^{1-\vartheta}\|\psi\|_{L^2(\pa E)}^{\vartheta}\bigl(\|\varphi\|_{L^2(\pa E)}+\|\psi\|_{L^2(\pa E)}\bigr) \,,
\end{eqnarray}
with $\vartheta\in(0,1)$  depending only on $p>N-1$. In the last inequality we used a well-known interpolation result, see for instance \cite[Theorem 5.2]{AdamsFournier}. 

\noindent {\it Step 2.} The  previous estimate does not allow to conclude directly, but we have to rely on the
 following iteration procedure.
Fix any number
\begin{equation}\label{K}
K>2
\end{equation}
 and assume that  $\eta\in (0,1)$ is such that 
\begin{equation}\label{eta}
C_2\eta(1+2C_1)\leq\e_0,\qquad2C_5\eta^\vartheta K\leq\delta\,.
\end{equation}
Given $\psi$, we set $\varphi_0=\psi$  and we denote by  $\sigma^1$ the vector defined as in  \eqref{uno}.  We set $F_1:=F-\sigma^1$  and denote by $\varphi_1$  the function   such that $\pa F_1=\{x+\varphi_1(x)\nu(x):x\in \pa E\}$.  As before, $\varphi_1$  satsfies 
$$
x+\varphi_0(x)\nu(x)-\sigma^1=y+\varphi_1(y)\nu(y)\,.
$$
Since $\|\psi\|_{W^{2,p}(\pa E)}\leq\eta$ and  $\|\psi\|_{L^{2}(\pa E)}\leq\eta$, by  \eqref{due}, \eqref{quattro},   and \eqref{eta} we have
\begin{equation}\label{undici}
\|\varphi_1\|_{W^{2,p}(\pa E)}\leq C_2\eta(1+C_1)\leq 1\,.
\end{equation}
Using again that  $\|\psi\|_{W^{2,p}(\pa E)}\leq\eta\leq1$,  by \eqref{dieci}  we obtain
$$
\Bigl|\int_{\pa E}\varphi_1(y)\nu(y)\,d\H(y)\Bigr|\leq C_5\|\varphi_0\|_{L^2(\pa E)}^{\vartheta}\bigl(\|\varphi_1\|_{L^2(\pa E)}+\|\varphi_0\|_{L^2(\pa E)}\bigr)\,.
$$
As for the last term we have  $\|\varphi_0\|_{L^2(\pa E)}\leq\eta$. 
 We now distinguish two cases. 
 
 \noindent If  $\|\varphi_0\|_{L^2(\pa E)}\leq K\|\varphi_1\|_{L^2(\pa E)}$,  from the previous inequality and \eqref{eta} we get 
$$
\Bigl|\int_{\pa E}\varphi_1(y)\nu(y)\,d\H(y)\Bigr|\leq C_5\eta^{\vartheta}\bigl(\|\varphi_1\|_{L^2(\pa E)}+\|\varphi_0\|_{L^2(\pa E)}\bigr)\leq 2C_5\eta^{\vartheta}K\|\varphi_1\|_{L^2(\pa E)}\leq\delta\|\varphi_1\|_{L^2(\pa E)}
$$
and, thus, the conclusion  follows with $\sigma=\sigma^1$. 

\noindent In the other case
\begin{equation}\label{dodici}
\|\varphi_1\|_{L^2(\pa E)}\leq\frac{\|\varphi_0\|_{L^2(\pa E)}}{K}\leq\frac{\eta}{K}\leq \eta\,.
\end{equation}
We repeat the whole procedure: denote by $\sigma^2$  the vector defined as in  \eqref{uno} with $\psi$ replaced by $\varphi_1$,  set $F_2:=F_1-\sigma^2=F-\sigma^1-\sigma^2$ and  consider the corresponding  $\varphi_2$. Then $\varphi_2$  satisfies
$$
z+\varphi_2(z)\nu(z)=y+\varphi_1(y)\nu(y)-\sigma^2=
x+\varphi_0(x)\nu(x)-\sigma^1-\sigma^2\,.
$$

Since 
\begin{align*}
\|\vphi_0\|_{W^{2,p}(\pa E)}+|\sigma_1+\sigma_2|& \leq \eta+C_1\eta+C_1\|\vphi_1\|_{L^2(\pa E)}\\
&\leq \eta+C_1\eta\Bigl(1+\frac1K\Bigr)\leq C_2\eta(1+2C_1)\leq\e_0\,,
\end{align*}
the map $T_{\vphi_0}(x):=\pi(x+\vphi_0(x)\nu(x)-(\sigma^1+\sigma^2))$ is a diffeomorphism thanks to \eqref{transone}.
Thus, by applying  \eqref{quattro} with $\sigma=\sigma_1+\sigma_2$, and \eqref{due}, \eqref{dodici}, \eqref{K},  \eqref{eta} implies
$$
\|\varphi_2\|_{W^{2,p}(\pa E)}\leq C_2\bigl(\|\varphi_0\|_{W^{2,p}(\pa E)}+|\sigma^1+\sigma^2|\bigr)\leq C_2\eta\Bigl(1+C_1+\frac{C_1}{K}\Bigr)\leq 1\,,
$$
analogously to \eqref{undici}. On the other hand, since by \eqref{undici}, \eqref{dodici}, and \eqref{due} 
$$
\|\vphi_1\|_{W^{2,p}(\pa E)}+\sigma_2\leq C_2\eta(1+C_1)+C_1\frac{\eta}{K}\leq 
C_2\eta(1+2C_1)\leq \e_0\,,
$$
also the map $T_{\vphi_1}(x):=\pi(x+\vphi_1(x)\nu(x)-\sigma^2)$ is a diffeomorphism
satisfying \eqref{transone} and \eqref{transtwo}. Therefore,
arguing as before, we obtain
$$
\Bigl|\int_{\pa E}\varphi_2(y)\nu(y)\,d\H(y)\Bigr|\leq C_5\|\varphi_1\|_{L^2(\pa E)}^{\vartheta}\bigl(\|\varphi_2\|_{L^2(\pa E)}+\|\varphi_1\|_{L^2(\pa E)}\bigr)\,.
$$
Since $\|\varphi_1\|_{L^2(\pa E)}\leq\eta$  by \eqref{dodici}, if $\|\varphi_1\|_{L^2(\pa E)}\leq K\|\varphi_2\|_{L^2(\pa E)}$  the conclusion follows  with $\sigma=\sigma^1+\sigma^2$. Otherwise,  we iterate the procedure observing that 
$$
\|\varphi_2\|_{L^2(\pa E)}\leq\frac{\|\varphi_1\|_{L^2(\pa E)}}{K}\leq\frac{\|\varphi_0\|_{L^2(\pa E)}}{K^2}\leq\frac{\eta}{K^2}\,.
$$
This construction leads to  three (possibly finite) sequences  $\sigma_n$,  $F_n$,  and $\varphi_n$  such that
$$
\begin{cases}
F_n=F-\s^1-\dots-\s^n, \qquad |\sigma^n|\leq\frac{C_1\eta}{K^{n-1}}\,, & \cr
\|\varphi_n\|_{W^{2,p}(\pa E)}\leq C_2\bigl(\|\varphi_0\|_{W^{2,p}(\pa E)}+|\sigma^1+\dots+\sigma^n|\bigr)\leq C_2\eta(1+2C_1)\,, & \cr
\|\varphi_n\|_{L^2(\pa E)} \leq\frac{\eta}{K^n}\,,  & \cr
\partial F_n=\{x+\varphi_n(x)\nu(x):\,x\in\partial E\}\,.
\end{cases}
$$
If for some $n$ we have $\|\varphi_{n-1}\|_{L^2(\pa E)}\leq K\|\varphi_n\|_{L^2(\pa E)}$,
the construction stops, since, arguing as before,  
$$
\Bigl|\int_{\pa E}\varphi_n(y)\nu(y)\,d\H(y)\Bigr|\leq\delta\|\varphi_n\|_{L^2(\pa E)}
$$
 and conclusion follows with  $\sigma=\s^1+\dots+\s^n$ and $\varphi=\varphi_n$. 
 Otherwise,  the  iteration continues indefinitely and we reach the conclusion with 
$$
\sigma=\sum_{n=0}^\infty\sigma^n, \qquad\varphi=0\,, 
$$
which means that  $F=E+\sigma$.

\end{proof}

We are now ready to prove the main result of this section.

\begin{theorem}\label{th:c2min}
Let $p>\max\{2, N-1\}$ and let $E$ be a regular  critical set for $J$  with positive second variation. Then there exist $\delta>0$, $C_0>0$ such that 
$$
J(F)\geq J(E)+C_0\bigl(\alpha(E,F)\bigr)^2\,,
$$
whenever  $F\subset \T$ satisifes $|F|=|E|$ and $\pa F=\{x+\psi(x)\nu(x):\, x\in \pa E\}$ for some $\|\psi\|_{W^{2,p}(\pa E)}\leq \delta$.
 \end{theorem}
\begin{proof} 
Since all estimates will depend only on  $\|\psi\|_{W^{2,p}(\pa E)}$, by an approximation argument we may assume that $\psi$ is of class $C^{\infty}$.
 Moreover, since different sets are involved we employ the full notation for the normal vectors.
 
  \noindent {\it Step 1 .}  We claim that there exists $\delta_1>0$ such that
 if $F=\{x+\psi(x)\nu
(x):\, x\in \pa E\}$ with $|F|=|E|$ and $\|\psi\|_{W^{2,p}(\pa E)}\leq \delta_1$, then
\beq\label{c2unif}
\inf\Bigl\{\pa^2J(F)[\vphi]:\, \vphi\in \Ht(\pa F)\,, \|\vphi\|_{H^1(\pa F)}=1\,, \Bigl|\int_{\pa F}\vphi\nu^F\, d\H\Bigr|\leq\delta_1\Bigr\}\geq\frac{m_0}2\,,
\eeq
where $m_0$ is defined in \eqref{c0}. 
We argue by contradiction assuming that there exist a sequence $F_h=\{x+\psi_h(x)\nu
(x):\, x\in \pa E\}$ with $|F_h|=|E|$ and $\|\psi_h\|_{W^{2,p}(\pa E)}\to 0$ and a sequence  $\vphi_h\in \Ht(\pa F_h)$, with $\|\vphi_h\|_{H^1(\pa F_h)}=1$ and 
$$
\int_{\pa F_h}\vphi_h\nu^{F_h}\, d\H\to 0
$$
such that 
\beq\label{liminf}
\pa^2J(F_h)[\vphi_h]<\frac{m_0}{2}\,.
\eeq
Consider a family $\Phi_h$ of diffeomorphisms from $E$ to $F_h$ converging  to the identity in $W^{2,p}(\pa E)$, which exists by the convergence of $\psi_h$ to $0$. Set 
$$
a_h:=\medint_{\pa E}\vphi_h\circ \Phi_h\, d\H\qquad\text{and}\qquad
\tilde\vphi_h:=\vphi_h\circ \Phi_h-a_h:
$$
 since $\nu^{ F_h}\circ \Phi_h\to \nu$ in $C^{0,\alpha}(\pa E)$ and a similar convergence holds for the tangential vectors, one easily checks that for those $i$ for which $\nu
\cdot \e_i\not\equiv0$ we have
 $$
 \int_{\pa E}\tilde\vphi_h\nu\cdot\e_i\, d\H\to 0\,,
 $$ 
 so that, using \eqref{projtperp}, 
\beq\label{proj1}
\| \pi_{T^\perp (\pa E)}(\tilde \vphi_h)\|_{H^1(\pa E)}\to 1\,.
\eeq
Moreover, the second fundamental forms and the functions $v_{F_h}$ (see \eqref{vE}) satisfy
\beq\label{proj2}
B_{\pa F_h}\circ \Phi_h\to B_{\pa E}\ \text{in  $L^p(\pa E)$},\qquad v_{F_h}\to v_E\ \text{in $C^{1,\beta}(\T)$ for all $\beta<1$.}
\eeq
Indeed, the first convergence follows immediately by the $W^{2,p}$ convergence of $F_h$ to $E$, while the second one is implied by \eqref{stimauniv}.

We now show that
\beq\label{proj3}
\int_{\partial F_h}\!\int_{\partial F_h}\!\!\!G(x,y)\vphi_h(x)\vphi_h(y)\, d\H d\H-
\int_{\partial E}\!\int_{\partial E}\!\!G(x,y)\tilde\vphi_h(x)\tilde\vphi_h(y)\, d\H d\H \to 0
\eeq
as $h\to\infty$, which in turn is equivalent to proving that 
\beq\label{differenze}
\int_\T\bigl(|\nabla z_{h}|^2-|\nabla \tilde z_{h}|^2\bigr)\, dx\to 0\,,
\eeq
where 
$$
-\Delta z_h=\mu_h:=\vphi_h\H\lfloor\pa F_h\,,\qquad -\Delta \tilde z_h=\tilde\mu_h:=\tilde\vphi_h\H\lfloor\pa E\,,
$$
see \eqref{nonloc}.
Clearly, it is enough to show that $\mu_h-\tilde \mu_h\to 0$ strongly in $H^{-1}(\T)$. 
Indeed, from this convergence it would follow that $z_h-\tilde z_h\to 0$ in $H^1(\T)$ and, in turn, that  \eqref{differenze} holds, since both sequences $z_h$ and 
$\tilde z_h$ are bounded in $H^1(\T)$.    To prove that $\mu_h-\tilde \mu_h\to 0$ strongly in $H^{-1}(\T)$
we fix $w\in H^1(\T)\cap C^1(\T)$,  with $\|w\|_{H^1(\T)}\leq 1$. Then, denoting by $J_{N-1}(d^{\pa E}\Phi_h)$ the Jacobian of $\Phi_h$ on $\pa E$, 
\begin{align*}
\langle \mu_h-\tilde \mu_h, w\rangle&= \int_\T w\, d(\mu_h-\tilde \mu_h)\\
&=
\int_{\pa E}\Bigl[w(\Phi_h(x))\tilde\vphi_h(x)J_{N-1}(d^{\pa E}\Phi_h)(x)-w(x)\tilde\vphi_h(x)\Bigr]\, d\H+a_h\int_{\pa F_h}w\, d\H\\
&=
\int_{\pa E}\tilde\vphi_h(x)\bigl[w(\Phi_h(x))-w(x)\bigr]J_{N-1}(d^{\pa E}\Phi_h)(x)\, d\H\\
&\quad+
\int_{\pa E}\bigl[J_{N-1}(d^{\pa E}\Phi_h)(x)-1\bigr]w(x)\tilde\vphi_h(x)\, d\H+a_h\int_{\pa F_h}w\, d\H\,.
\end{align*}
Therefore we can estimate
\begin{align*}
|\langle \mu_h-\tilde \mu_h, w\rangle|& \leq \|J_{N-1}(d^{\pa E}\Phi_h)\|_{L^\infty(\pa E)}\cdot \|\tilde \vphi_h\|_{L^2(\pa E)}\cdot \|w\circ\Phi_h-w\|_{L^2(\pa E)}\\
&\quad+
c\|J_{N-1}(d^{\pa E}\Phi_h)-1\|_{\infty}\cdot \|\tilde\vphi_h\|_{L^2(\pa E)}\cdot \|w\|_{H^1(\T)}+c|a_h|\|w\|_{H^1(\T)}\,.
\end{align*}
Arguing as in the proof of \eqref{come}, we have
\begin{align*}
\|w\circ\Phi_h-w\|_{L^2(\pa E)}^2& = \int_{\pa E}|w(x+\psi_h(x)\nu
(x))-w(x)|^2 d\H\\
&\leq \int_{\pa E}|\psi_h|^2\int_0^1|\nabla w(x+t\psi_h(x)\nu
(x))|^2\, dtd\H\leq C\|\psi_h\|^2_{\infty}\|\nabla w\|^2_{L^2(\T)}\,.
\end{align*}
Combining all the above estimates, we may conclude that 
$$
\|\mu_h-\tilde \mu_h\|_{H^{-1}(\T)}\leq C\Bigl(\|\psi_h\|_{L^\infty(\pa E)}+\|J_{N-1}(d^{\pa E}\Phi_h)-1\|_{L^\infty(\pa E)}+|a_h|\Bigr)\to 0
$$
thus proving \eqref{proj3}.
 From \eqref{proj1}, \eqref{proj2}, and \eqref{proj3}, recalling that $p>\max\{2, N-1\}$ and using the Sobolev Embedding to
 show that 
 $$
 \int_{\pa F_h}|B_{\pa F_h}|^2\vphi_h^2\, d\H-\int_{\pa E}|B_{\pa E}|^2\tilde \vphi_h^2\, d\H\to 0\,,
 $$
   it follows that all terms 
in the expression \eqref{pa2J} of $\pa^2J(F_h)[\vphi_h]$ are asympotically close to the corresponding terms of $\pa^2J(E)[\tilde\vphi_h]$. Hence $\pa^2J(F_h)[\vphi_h]-\pa^2J(E)[\tilde\vphi_h]\to 0$. Since 
$\pa^2J(E)[\tilde \vphi_h]-\pa^2J(E)[(\tilde \vphi_h)^\perp]\to 0$ and $\|(\tilde \vphi_h)^\perp\|_{H^1(\pa E)}\to 1$,  from Lemma~\ref{coerc} we get a  contradiction to \eqref{liminf}.

\noindent{\it Step 2 .} Let us fix $F$ so that $\|\psi\|_{W^{2,p}(\pa E)}\leq\delta_2<\delta_1$, where $\delta_2$ is to be chosen, and consider the field $X$ and the flow $\Phi$ constructed in Theorem~\ref{prop:connect}. Replacing $F$ by a $F-\sigma$ for some $\sigma\in \R^N$, if needed, thanks to   
Lemma~\ref{lm:italiano} we may assume 
\beq\label{s20}
\Bigl|\int_{\pa E}\psi\nu
\, d\H\Bigr|\leq \frac{\delta_1}2\|\psi\|_{L^2(\pa E)}\,.
\eeq
We claim that 
\beq\label{s21}
\Bigl|\int_{\pa E_t}(X\cdot \nu^{E_t})\nu^{E_t}\, d\H\Bigr|\leq \delta_1\|X\cdot \nu^{E_t}\|_{L^2(\pa E_t)}
\eeq
for all $t\in [0,1]$.  To this aim, we write
\begin{eqnarray*}
&&\int_{\pa E_t}(X\cdot \nu^{E_t})\nu^{E_t}\, d\H \\
&&\quad=\int_{\pa E}(X(\Phi(x,t))\cdot \nu^{E_t}(\Phi(x,t)))\nu^{E_t}(\Phi(x,t))J_{N-1}(d^{\pa E}\Phi(\cdot, t))(x)\, d\H \\
&&\quad= \int_{\pa E}(X(\Phi(x,t))\cdot \nu(x))\nu(x)\, d\H+R_1\\
&&\quad =\int_{\pa E}(X(x)\cdot \nu(x))\nu(x)\, d\H+R_1+R_2\\
&&\quad  =\int_{\pa E}\psi(x)\nu(x)\, d\H+R_1+R_2+R_3\,.
\end{eqnarray*}
Fix $\e>0$. Recalling \eqref{connect7}, \eqref{connecttwo}, \eqref{connect1.5}, and \eqref{connect3},  we have 
$$
\int_{\pa E}|X(\Phi(x,t))|\, d\H\leq C\|\psi\|_{L^2(\pa E)}\,.
$$
From this inequality, observing that by \eqref{campo0}
$$
\|\nu
-\nu^{E_t}(\Phi(\cdot, t))\|_{L^{\infty}(\pa E)}\,, \qquad
\|J_{N-1}(d^{\pa E}\Phi(\cdot, t))-1\|_{L^{\infty}(\pa E)} 
$$
are arbitrarily small, and recalling \eqref{connect0} and \eqref{campo3}  we deduce that
$$
|R_1|+|R_2|+|R_3| \leq \e\|\psi\|_{L^2(\pa E)}\,,
$$
provided that $\delta_2$ is sufficiently small.
 This proves that  
$$
\Bigl|\int_{\pa E_t}(X\cdot \nu^{E_t})\nu^{E_t}\, d\H \Bigr| \leq \Bigl|\int_{\pa E}\psi\nu
\, d\H\Bigr|+ \e\|\psi\|_{L^2(\pa E)}\leq
 \Bigl(\frac{\delta_1}2+\e\Bigr)\|\psi\|_{L^2(\pa E)}\,,
$$
where we used also \eqref{s20}. A similar argument shows that
\beq\label{s22}
\|X\cdot \nu^{E_t}\|_{L^2(\pa E_t)}\geq (1-\e)\|\psi\|_{L^2(\pa E)}\,,
\eeq
and thus  \eqref{s21} follows, if $\e$ and, in turn, $\delta_2$ are suitably chosen.

Recalling \eqref{eq:J''}, \eqref{pa2J},  the fact that $E$ is a critical set for $J$ and that $\Div X=0$ in a neighborhood of $\pa E$, we can write
\begin{align*}
&J(F)-J(E) =J(E_1)-J(E)= \frac12\int_0^1(1-t)\frac{d^2}{dt^2}J( E_t)\, dt\\
&=\frac12\int_0^1(1-t)\biggl(\pa^2J( E_t)[ X\cdot \nu^{E_t}]
-\int_{\pa  E_t}(4\gamma v_{ E_t}+H_t)\Div_{\tau_t}\bigl(X_{\tau_t}(X\cdot \nu^{E_t})\bigr)\,d\H\biggr)\, dt\,,
\end{align*}
where $\Div_{\tau_t}$ stands for the tangential divergence on $\pa E_t$, we set
$ X_{\tau_t}:= X-( X\cdot\nu^{ E_t})\nu^{ E_t}$, and $H_t$ is the sum of principal 
curvatures of $\pa E_t$.
By \eqref{s21} and \eqref{c2unif}, we obtain that
\begin{align}
J(F)-J(E)& \geq \frac{m_0}{4}\int_0^1(1-t)\|X\cdot \nu^{E_t}\|^2_{H^1(\pa  E_t )}\, dt\nonumber \\
&\quad -\frac12
\int_0^1(1-t)\int_{\pa  E_t}(4\gamma v_{E_t}+H_t)\Div_{\tau_t}\bigl(X_{\tau_t}(X\cdot \nu^{ E_t})\bigr)\, d\H dt\,.\label{c2min4}
\end{align}
  We claim that 
\beq\label{c2min7}
I_t:=\biggl|\int_{\pa E_t}(4\gamma v_{ E_t}+H_t)\Div_{\tau_t}\bigl(X_{\tau_t}(X\cdot \nu^{ E_t})\bigr)\, d\H\biggr|\leq \frac{m_0}{4} \|X\cdot \nu^{ E_t}\|^2_{H^1(\pa E_t )}
\eeq
for all $t\in [0,1]$, provided that $\delta_2$ is sufficiently small.

Indeed, recalling that $E$ satisfies \eqref{ELJ}, we get
\begin{align}
I_t&=\biggl|\int_{\pa  E_t}\bigl[(4\gamma v_{E_t}+H_t)-\lambda\bigr]\Div_{\tau_t}\bigl( X_{\tau_t}( X\cdot \nu^{ E_t})\bigr)\, d\H\biggr|\nonumber\\
&\leq \|(4\gamma v_{ E_t}+H_t)-\lambda\|_{L^p(\pa  E_t)}
\|\Div_{\tau_t}\bigl( X_{\tau_t}(X\cdot \nu^{E_t})\bigr)\|_{L^{\frac{p}{p-1}}(\pa  E_t)} \,.  
  \label{c2min8}
\end{align}
Observe that, given $\e>0$, if $\delta_2$ is sufficiently small the first norm on the right-hand side of \eqref{c2min8} can be taken smaller than $\e$. Hence, using Lemma~\ref{lm:contazzi} below we get 
\begin{align*}
I_t & \leq c\e\Bigl[\|D_{\tau_t}\bigl(X_{\tau_t}\bigr)\|_{L^2(\pa E_t)}\|X\cdot \nu^{ E_t}\|_{L^{\frac{2p}{p-2}}(\pa E_t)}+\|X_{\tau_t}\|_{L^{\frac{2p}{p-2}}(\pa E_t)}\|D_{\tau_t}\bigl(X\cdot \nu^{ E_t}\bigr)\|_{L^{2}(\pa E_t)}\Bigr]\\
& \leq c\e\|X\cdot \nu^{E_t}\|_{H^{1}(\pa E_t)}\| X\cdot \nu^{ E_t}\|_{L^{\frac{2p}{p-2}}(\pa E_t)}\,.
\end{align*}
Since $p>\max\{2, N-1\}$, from the Sobolev Embedding Theorem we obtain 
$$
I_t\leq c\e \|X\cdot \nu^{ E_t}\|^2_{H^{1}(\pa E_t)}\,,
$$ 
hence \eqref{c2min7} follows.

We now  observe that from \eqref{c2min4} and \eqref{c2min7} we have 
\begin{align*}
J(F)& \geq J(E)+\frac{m_0}{8}\int_0^1(1-t)\| X\cdot \nu^{E_t}\|^2_{H^{1}(\pa  E_t)}\, dt\geq
J(E)+
\frac{m_0}{8}\int_0^1(1-t)\| X\cdot \nu^{ E_t}\|^2_{L^2(\pa E_t)}\, dt\,. 
\end{align*}
Recalling \eqref{s22}, we finally get 
$$
J(F) \geq J(E)+\frac{m_0}{32}\|\psi\|^2_{L^2(\pa E)}\geq J(E)+C_0|E\triangle F|^2\,.
$$
This concludes the proof of the theorem.
\end{proof}


\section{$W^{2,p}$-local minimality implies  $L^1$-local minimality}\label{sec:locmin}

We start by proving the following simple lemma.
\begin{lemma}\label{lm:noname}
Let $E\subset\T$ be of class $C^2$ and let $F\subset\T$ be a set of finite perimeter. Then there exists $C=C(E)>0$ such that 
$$
P_\T(F)-P_\T(E)\geq -C|E\triangle F|\,.
$$
\end{lemma}
\begin{proof}
Let   $X\in C^1(\T; \R^N)$ be a  vector field such that $\|X\|_{\infty}\leq 1$ and $X=\nu^E$ on $\pa E$. Then,
\begin{align*}
P_\T(F)-P_\T(E)&\geq \int_{\pa^* F}X\cdot\nu^F\, d\H- \int_{\pa E}X\cdot\nu^E\, d\H\\
&=\int_F\Div X\, dx-\int_E\Div X\, dx\geq
-C|E\triangle F|\,,\end{align*}
where $C:=\|\Div X\|_{\infty}$.
\end{proof}
Theorem~\ref{th:inftymin} below shows that if $E$ is a smooth isolated $W^{2,p}$-local minimizer of $J$, in the sense of Theorem~\ref{th:c2min}, then
$E$ is also a minimizer among all competitors which are sufficiently close in the Hausdorff distance. Some points  in the proof are adapted from \cite{white}, see also \cite{FM}. In the proof the theorem we will make use of  an important regularity result concerning sequences of $\omega$-minimizers of the area functional. This is essentially contained in \cite{Alm} (see also \cite{SS, Ta2, white}).  
\begin{theorem}\label{th:white}
Let $E_h\subset\T$ be  a sequence of $\omega$-minimizers of the area functional such that
$$
\sup_h\P(E_h)<\infty \qquad\text{and}\qquad \chi_{E_h}\to \chi_E\quad\text{in $L^1(\T)$}
$$
for some set $E$   of class $C^2$. Then,  for $h$ large enough $E_h$ is of class $C^{1, \frac12}$  and 
$$
\pa E_h=\{x+\psi_h(x)\nu(x):x\in \pa E\}\,,
$$
with $\psi_h\to 0$ in $C^{1, \alpha}(\pa E)$ for all $\alpha\in (0, \frac12)$.
\end{theorem}
Recalling that $d$ denotes the signed distance to a set $E$, we define for all $\de\in \R$
$$\mathcal{I}_{\de}(E)=\{x:d(x)<\de\}\;.$$
We are now ready to state the $L^{\infty}$-local minimality result. 
\begin{theorem}\label{th:inftymin}
Let $E\subset \T$ be a smooth set and $p>1$. Assume that there exists $\delta>0$ such that 
\beq\label{infty1}
J(F)\geq J(E)
\eeq
for all  $F\subset \T$, with  $|F|=|E|$ and such that 
$\pa F=\{x+\psi(x)\nu(x):\, x\in \pa E\}$ for some function $\psi$ with $\|\psi\|_{W^{2,p}(\pa E)}\leq \delta$.
Then there exists $\delta_0>0$ such that \eqref{infty1} holds for all $F\subset \T$ of finite perimeter, with $|F|=|E|$ and
$ \mathcal I_{-\delta_0}(E)\subset F\subset \mathcal I_{\delta_0}(E)$.
\end{theorem}
\begin{proof}
We argue by contradiction assuming that there exist  two sequences $\de_h\to 0$ and  $E_h\subset\T$ such that $|E_h|=|E|$, 
$\mathcal I_{-\de_h}(E)\subset E_h\subset \mathcal{I}_{\de_h}(E)$, and
$$
J(E_h)<J(E)
$$
for all $h$. For every $h$ let $F_h$ be a minimizer of the  penalized obstacle problem
\beq\label{infty3}
\min\{J(F)+\Lambda \bigl||F|-|E|\bigr|:\,\mathcal{I}_{-\de_h}(E)\subset F\subset \mathcal{I}_{\de_h}(E)\}\,,
\eeq
where $\Lambda>1$ will be chosen later. Clearly,
\beq\label{infty9}
J(F_h)\leq J(E_h)<J(E)\,.
\eeq
We split the proof into four steps.

\noindent{\it Step 1.}  We claim that for $\Lambda>0$ sufficiently large
\beq\label{infty8}
|F_h|=|E|\,. 
\eeq
Indeed, assume by contradiction that
$|F_h|\neq|E|$. We consider the case $|F_h|<|E|$.  We define 
$$
\widetilde F_h=F_h\cup \mathcal I_{\tau_h}(E)
$$
for some $\tau_h\in (-\de_h, \de_h)$ such that $|\widetilde F_h|=|E|$. Set $\nu:=\nabla d$. Since $\pa^*\widetilde F_h$ can be decomposed in three disjoint parts, one contained in $\pa^* F_h\setminus \partial I_{\tau_h}(E)$, another contained in 
$\partial I_{\tau_h}(E)\setminus \pa^* F_h$, and the third one given by 
$\{x\in \pa^*F_h\cap \pa  I_{\tau_h}(E):\, \nu^{F_h}(x)=\nu^{\mathcal I_{\tau_h}(E)}(x)\}$, and since $\nu^{\mathcal I_{\tau_h}(E)}= \nu$, we have 
$$
 \P(\widetilde F_h)-\P(F_h)\leq \int_{\pa^*\widetilde F_h}\nu\cdot \nu^{\widetilde F_h}\, d\H-
 \int_{ \pa^*F_h}\nu\cdot \nu^{ F_h}\, d\H\,.
$$
Hence, also by Lemma~\ref{lm:v},
\begin{align}
J(\widetilde F_h)&+\Lambda \bigl||\widetilde F_h|-|E|\bigr|-
J(F_h)-\Lambda \bigl||F_h|-|E|\bigr|\nonumber\\
&= \P(\widetilde F_h)-\P(F_h)
+\gamma\int_{\T}\bigl(|\nabla v_{\widetilde F_h}|^2-|\nabla v_{ F_h}|^2\bigr)\, dx-
\Lambda \bigl(|\widetilde F_h|-|F_h|\bigr)\nonumber\\
&\leq \int_{\pa^*\widetilde F_h}\nu\cdot \nu^{\widetilde F_h}\, d\H-
 \int_{ \pa^*F_h}\nu\cdot \nu^{ F_h}\, d\H+(\gamma C-\Lambda)\bigl(|\widetilde F_h|-|F_h|\bigr)\nonumber\\
 &\leq \int_{\widetilde F_h\triangle F_h}|\Div\nu|\, dx+(\gamma C-\Lambda)\bigl(|\widetilde F_h|-|F_h|\bigr)\nonumber\\
 &\leq (\|\Div \nu\|_{\infty}+\gamma C-\Lambda)\bigl(|\widetilde F_h|-|F_h|\bigr)\,. \label{infty3bis}
\end{align}
Thus, if 
\beq\label{infty4}
\Lambda>\|\Div \nu\|_{\infty}+\gamma C
\eeq
 the  last term of the previous inequality is negative, thus contradicting the minimality of $F_h$. If $|F_h|>|E|$, we argue similarly.

\noindent{\it Step 2.} For any set $F$, we set $\mathcal K_h(F):=(F\cup \mathcal I_{-\delta_h}(E))\cap \mathcal I_{\delta_h}(E)$. We claim that $F_h$ solves the penalized problem (without obstacle)
\beq\label{infty5}
\min\{J(F)+\Lambda\bigl||F|-|E|\bigr|+2\Lambda|F\triangle \mathcal K_{h}(F)|:\, F\subset\T\}\,.
\eeq
Indeed, let $\widetilde J$ denote the functional in \eqref{infty5}. Writing $\mathcal K_h$ for $\mathcal K_h(F)$,  using Lemma~\ref{lm:v} and arguing as in 
\eqref{infty3bis}, we obtain by the minimality of $F_h$
\begin{align*}
\widetilde J(F)&-\widetilde J(F_h)=J( \mathcal K_h)+
\Lambda\bigl| |\mathcal K_h|-|E|\bigr|- \bigl[J(F_h)+\Lambda\bigl||F_h|-|E|\bigr|\bigr]\\
&\quad +\bigl[\P(F)- \P(\mathcal K_h)]+
\gamma\int_\T|\nabla v_F|^2\, dx-\gamma\int_\T|\nabla v_{\mathcal K_h}|^2\, dx\\
&\quad +\Lambda\bigl(\bigl||F|-|E|\bigr|- \bigl||\mathcal K_h|-|E|\bigr|\bigr) +2\Lambda|F\triangle \mathcal K_h|\\
&\geq -\|\Div \nu\|_{\infty}|F\triangle \mathcal K_h|-\gamma C|F\triangle \mathcal K_h|+\Lambda|F\triangle \mathcal K_h|>0\,,
\end{align*}
where in the last inequality we used \eqref{infty4}.

\noindent{\it Step 3.} We claim that for $h$ large enough $F_h$ is of class $C^{1,\frac12}$ and 
$$
\pa F_h=\{x+\psi_h(x)\nu(x):x\in \pa E\}\,,
$$
for some $\psi_h$ such that  $\psi_h\to 0$ in $C^{1, \alpha}(\pa E)$ for all $\alpha\in (0, \frac12)$. To this aim we observe that $F_h$  solves 
\eqref{infty5}, thus it is a $4\Lambda$-minimizer of the area functional. By Theorem~\ref{th:white} the claim follows.

\noindent{\it Step 4.} We claim that $\psi_h\to 0$ in $W^{2, p}(\pa E)$ for all $p>1$. To this aim, we first observe that since $F_h$ is a $C^1$ solution of the minimum problem \eqref{infty5}, a standard variation argument (see Step 2 of the proof of Proposition 7.41 in 
\cite{AFP}) yields 
\beq\label{infty5bis}
\sup_h\|H_{\pa F_h}\|_{L^{\infty}(\pa F_h)}\leq 4\Lambda\,,
\eeq
where, we recall, $H_{\pa F_h}$  denotes the sum of the principal curvatures of $\pa F_h$. Since the functions $\psi_h$ are equibounded in $C^{1,\alpha}$, the above estimate on the curvatures implies that for all $p>1$ the functions $\psi_h$ are equibounded in $W^{2,p}(\pa E)$, thanks to Remark~\ref{rm:lemmino}.
Recall now that, due to \eqref{infty8}, each $F_h$ is a  solutions of the obstacle problem \eqref{infty3} under the volume constraint. Since $F_h$ is of class $W^{2,p}$, we have that $H_{\pa F_h}=f_h$, where
\beq\label{infty5ter}
f_h:=
\begin{cases}
\lambda_h-4\gamma v_{F_h} & \text{in $A_h:=\pa F_h\cap \mathcal N_{\delta_h}(\pa E)$,}\\
\lambda-4\gamma v_E+\rho_h & \text{otherwise}\,,
\end{cases}
\eeq
 $\lambda_h$ and $\lambda$ are the volume constraint Lagrange multipliers  corresponding to $F_h$ and $E$, respectively, 
and  $\rho_h$ is a remainder term converging uniformly to  $0$.

We claim that
\beq\label{infty6}
 H_{\pa F_h}\bigl(\cdot+\psi_h(\cdot)\nu(\cdot)\bigr)\to H_{\pa E}(\cdot)\qquad\text{in $L^{p}( \pa E)$ for all $p>1$.}
 \eeq
To this aim, first observe  that  
\beq\label{recente3}
v_{F_h}\to v_{E} \quad\text{ in $C^{1}(\T)$}
\eeq
 by \eqref{stimauniv} and Lemma~\ref{lm:v}.
Moreover, from \eqref{infty5bis} we have that the sequence $\lambda_h$ is bounded. 

If $\H(A_h)\to 0$, where $A_h$ is defined in \eqref{infty5ter}, then \eqref{infty6} follows immediately.
Otherwise, (with no loss of generality) we have  $\H(A_h)\geq c>0$ and we argue as follows. By a compactness argument we may find a cylinder 
$C=B'\times(-L,L)$, where $B'\subset\R^{N-1}$ is a ball centered at the origin,  and functions $g_h$, $g\in W^{2,p}(B'; (-L,L))$ such that, upon rotating and relabeling the coordinate axes if necessary, we have $E\cap C=\{(x',x_n)\in B'\times(-L,L):\, x_n<g(x')\}$, 
\beq\label{infty7}
 F_h\cap C=\{(x',x_n)\in B'\times(-L,L):\, x_n<g_h(x')\}\,, \qquad\text{and}\quad \H(A_h\cap C)\geq c'>0
\eeq
for all $h$. Moreover, by Step 3 we also have
\beq\label{infty6bis}
g_h\to g \qquad\text{in $C^{1, \alpha}(\overline {B'})$ for all $\alpha\in (0,\tfrac12)$.}
\eeq
Denote by $A'_h$ the projection of $A_h\cap C$ over $B'$.  Then from \eqref{infty5ter} we have
\begin{align*}
&\lambda_h \H(A'_h)-4\gamma\int_{A'_h}v_{F_h}(x', g_h(x'))\, d\H(x')\nonumber \\
&+\lambda  \H(B'\setminus A'_h)-4\gamma \int_{B'\setminus A'_h}v_{E}(x', g(x'))\, d\H(x')+\omega_h\nonumber\\
&\qquad=\int_{B'}\Div\biggl(\frac{\nabla_{x'} g_h}{\sqrt{1+|\nabla_{x'} g_h|^2}}\biggr)\, d\H(x')=
\int_{\pa B'}\frac{\nabla_{x'} g_h}{\sqrt{1+|\nabla_{x'} g_h|^2}}\cdot \frac{x'}{|x'|}\, d\mathcal{H}^{N-2}
\end{align*}
with $\omega_h\to 0$.
Since by \eqref{infty6bis}
\begin{align*}
\int_{\pa B'}\frac{\nabla_{x'} g_h}{\sqrt{1+|\nabla_{x'} g_h|^2}}\cdot \frac{x'}{|x'|}\, d\mathcal{H}^{N-2}&\to 
\int_{\pa B'}\frac{\nabla_{x'} g}{\sqrt{1+|\nabla_{x'} g|^2}}\cdot \frac{x'}{|x'|}\, d\mathcal{H}^{N-2}\\
&=
\int_{B'}\Div\biggl(\frac{\nabla_{x'} g}{\sqrt{1+|\nabla_{x'} g|^2}}\biggr)\, d\H(x')\\
&=\lambda  \H(B')-4\gamma \int_{B'}v_{E}(x', g(x'))\, d\H(x')\,,
\end{align*}
recalling \eqref{recente3}, we conclude that 
$$
(\lambda_h-\lambda)\H(A'_h)\to 0\,.
$$
As   $\H(A'_h)\geq c''>0$ by \eqref{infty7}, we obtain \eqref{infty6}. In turn, by Lemma~\ref{lemmino} we conclude that 
$\psi_h\to 0$ in $W^{2,p}(\pa E)$ for all $p>1$.  Thus, by Theorem~\ref{th:c2min} and recalling \eqref{infty8}, we have that
$J(F_h)\geq J(E)$ for $h$ sufficiently large, a contradiction to \eqref{infty9}.
\end{proof}

We are now ready to prove the main result of the paper.

\begin{proof}[Proof of Theorem~\ref{th:l1min}.]
We argue by contradiction assuming that there exists a sequence $E_h\subset\T$, with $|E_h|=|E|$, such that $\alpha(E_h,E)\to 0$ and 
\beq\label{l1min1}
J(E_h)\leq J(E)+\frac{C_0}{4}\bigl(\alpha(E_h,E)\bigr)^2\,,
\eeq
where $C_0>0$ is the constant appearing in Theorem~\ref{th:c2min}.
By translating the sets if necessary, we may assume that  $\chi_{E_h}\to \chi_{E}$ in $L^1(\T)$. We now replace the sequence
$E_h$ with a sequence of minimizers $F_h$ of the following penalized functional
\beq\label{l1min2}
J(F)+\Lambda_1\sqrt{\bigl(\alpha(F,E)-\e_h\bigr)^2+\e_h}+\Lambda_2\bigl||F|-|E|\bigr|\,,
\eeq
where $\e_h:=\alpha(E_h,E)$, while the constants  $\Lambda_1$, $\Lambda_2$ will be chosen later.
Up to a subsequence we may assume that $\chi_{F_h}\to \chi_{F_0}$ in $L^1$, where   $F_0\subset\T$ is a minimizer of
$$
J(F)+\Lambda_1\alpha(F,E)+\Lambda_2\bigl||F|-|E|\bigr|
$$
and therefore, by translating $F_0$ and $F_h$ if necessary, also of
\beq\label{l1min3}
J(F)+\Lambda_1|E\triangle F|+\Lambda_2\bigl||F|-|E|\bigr|\,.
\eeq
Using Lemma~\ref{lm:noname} and arguing as in the proof of \eqref{infty3bis}, one can prove that if $\Lambda_1$ is sufficiently large (independently of $\Lambda_2$) the unique minimizer of \eqref{l1min3} is $E$. Thus, $F_0=E$. We now observe that the same argument used in the proof of Proposition~\ref{prop:EF} shows that if $\Lambda_2$ is sufficiently large then $|F_h|=|E|$ for all $h$.
Moreover, using Lemma~\ref{lm:v}, it  can be checked that for all $h$ the set $F_h$ is a $\Lambda$-minimizer of the area functional  for some $\Lambda>0$  independent of $h$.
Therefore, Theorem~\ref{th:white} yields that $F_h\to E$ in $C^{1,\alpha}$ for all $\alpha\in (0,\frac12)$. More precisely,
$$
\pa F_h=\{x+\psi_h(x)\nu(x):x\in \pa E\}\,,
$$
where $\psi_h\to 0$ in $C^{1, \alpha}(\pa E)$ for all $\alpha\in (0, \frac12)$. 

We show that $\psi_h\to 0$ in $W^{2, p}(\pa E)$ for all $p>1$. To this aim, observe that since $|F_h|=|E|$, each $F_h$ minimizes
the functional 
$$
J(F)+\Lambda_1\sqrt{\bigl(|F\triangle E|-\e_h\bigr)^2+\e_h}
$$ 
under the volume constraint $|F|=|E|$. 
Arguing as in Step 4 of the proof of Theorem~\ref{th:inftymin} we have that $\|H_{\pa F_h}\|_{L^\infty(\pa F_h)}$ is uniformly bounded,
hence the functions $\psi_h$ are equibounded in $W^{2,p}(\pa E)$ and the following Euler-Lagrange equation holds:
$$
H_{\pa F_h}=
\begin{cases}
\displaystyle \frac{\Lambda_1\bigl(\alpha(F_h,E)-\e_h\bigr)}{\sqrt{\bigl(\alpha(F_h,E)-\e_h\bigr)^2+\e_h}}\,{\rm sign}\,(\chi_{F_h}-\chi_E)
+\lambda_h-4\gamma v_{F_h} & \text{on $\pa F_h\setminus \pa E$,}\vspace{0.25cm}\\
\lambda-4\gamma v_E & \text{on $\pa F_h\cap \pa E$,}
\end{cases}
$$
where $\lambda_h$ and $\lambda$ are the Lagrange multipliers.
We claim that 
\beq\label{l1min5}
\e_h^{-1}\alpha(F_h, E)\to 1\,.
\eeq
 Indeed, if $|\alpha(F_h, E)-\e_h|\geq \sigma \e_h$ for some $\sigma>0$ and for infinitely many $h$, recalling \eqref{l1min1} and the fact that $F_h$ minimizes the functional \eqref{l1min2}, we have 
\begin{eqnarray*}
J(F_h)+\Lambda_1\sqrt{\sigma^2\e_h^2+\e_h} &\leq& J(E_h)+\Lambda_1\sqrt{\bigl(\alpha(E_h, E)-\e_h\bigr)^2+\e_h}=
J(E_h)+\Lambda_1\sqrt{\e_h}\\
&\leq& J(E)+\frac{C_0}{4}\bigl(\alpha(E_h,E)\bigr)^2+\Lambda_1\sqrt{\e_h}=
J(E)+\frac{C_0}{4}\e_h^2+\Lambda_1\sqrt{\e_h}\\
&\leq& J(F_h)+\frac{C_0}{4}\e_h^2+\Lambda_1\sqrt{\e_h}\,,
\end{eqnarray*}
where in the last inequality we have used the local minimality of $E$ with respect to $L^{\infty}$ perturbations proved in 
Theorem~\ref{th:inftymin}.  Since by the previous chain of inequalities we get that 
$$
\Lambda_1\sqrt{\sigma^2\e_h^2+\e_h}\leq \frac{C_0}{4}\e_h^2+\Lambda_1\sqrt{\e_h}\,,
$$
which is impossible for $h$ large, the claim is proved.  
In particular, 
$$
\biggl\|\frac{\Lambda_1\bigl(\alpha(F_h,E)-\e_h\bigr)}{\sqrt{\bigl(\alpha(F_h,E)-\e_h\bigr)^2+\e_h}}\,{\rm sign}\,(\chi_{F_h}-\chi_E)\biggr\|_{L^\infty(\pa F_h)}\to 0\,.
$$
Therefore, arguing  as in Step 4 of the proof of Theorem~\ref{th:inftymin}, we conclude that \eqref{infty6} holds thus proving
that $\psi_h\to 0$ in $W^{2, p}(\pa E)$ for all $p>1$.  We may now conclude: since $J(F_h)\leq J(E_h)$ by the minimality of $F_h$ and recalling \eqref{l1min5}, we have  that
$$
J(F_h)\leq J(E_h)\leq  J(E)+\frac{C_0}{4}\bigl(\alpha(E_h,E)\bigr)^2\leq 
 J(E)+\frac{C_0}{2}\bigl(\alpha(F_h,E)\bigr)^2
$$
for $h$ large.
This contradicts the minimality property proved in Theorem~\ref{th:c2min}.
\end{proof}
\begin{remark}\label{rm:recente1}
It is worth remarking that in the previous proof we did not use the second variation and we have in fact shown that any critical set $E$, for which the conclusion of Theorem~\ref{th:c2min} holds, satisfies also the conclusion of Theorem~\ref{th:l1min}. 
\end{remark}

We conclude this section by sketching the proof of the link between local minimizers of $J$ and of the Ohta-Kawasaki energy.
\begin{proof}[Proof of Theorem~\ref{th:OK}.]
We start by observing that the classical Modica-Mortola result (see \cite{Mo}; see also \cite{DalMaso} for the definition and properties of 
$\Gamma$-convergence) and the continuity of the non-local term with respect to the $L^1$ convergence of $u$ imply the  
$\Gamma$-convergence of 
$\mathcal E_\e$  to $\frac{16}3 \mathcal E$, where $\mathcal E$ is the functional defined in \eqref{egamma}.
The conclusion follows from the $L^1$-local minimality of $E$ proved in Theorem~\ref{th:l1min}, arguing as in 
\cite[Proposition 3.2]{CS2}.
\end{proof}
\begin{remark}
A careful inspection of the proof of \cite[Proposition 3.2]{CS2} shows that the radius $\de$ in the local minimality condition is uniform throughout the family $\{u_\e\}_{\e<\e_0}$ and depends only on the local minimality radius of the set $E$ appearing in
 Definition~\ref{def:locmin}. 
\end{remark}

\section{Application:  minimality of lamellae}\label{appl}
In this section we deal with global (and local) minimality of lamellar configurations. To this aim, for a given volume fraction parameter $m\in (-1,1)$ we denote by $u_L$ the one-strip lamellar configuration corresponding to the set $L:=\mathbb T^{N-1}\times[0,\frac{m+1}{2}]$ and by $\mathcal L_m
$ the collection of all sets which may be obtained from $L$ by translations and relabeling of coordinates.
\begin{theorem}\label{caffe}
Assume that  $L$ is the unique, up to translations and relabeling of coordinates, global minimizer of the periodic isoperimetric problem.  Then the same set is also the unique global minimizer of the non local functional \eqref{J},  provided $\gamma$ is sufficiently small.
\end{theorem}
\begin{proof}
We argue by contradiction assuming that there exist a sequence $\gamma_h\to0$ and a sequence of global minimizers $E_h$ of
$$
\min\{J_h(E):\,E\subset\T,\,|E|=|L|\}\;,
$$
where $J_h(E):=\P(E)+\gamma_h\int_{\T}|\nabla v_E|^2\,dx$, 
such that $E_h\not\in\mathcal L_m
$ for all $h$. Up to a subsequence we have $E_h\to E$ in $L^1$ and by the (easily proved) $\Gamma$-convergence of $J_h$ to the perimeter functional as $\gamma_h\to0$ we have that $E$ is a global minimizer of the periodic isoperimetric problem, thus by assumption $E\in \mathcal L_m
$. Without loss of generality we may assume $E=L$. 

To begin with, the second variation of $J_h$ at $L$ in \eqref{pa2J} reduces to
\[\begin{split}
\pa^2J_h(L)[\varphi]&= \int_{\pa L}|D_\tau\varphi|^2\, d\H+8\gamma_h \int_{\pa L}  \int_{\pa L}G(x,y)\varphi(x)\varphi(y)d\H(x)\,d\H(y) \\
& \quad +4\gamma_h\int_{\pa L}\partial_{\nu}v\varphi^2\,d\H\\
&\geq \int_{\pa L}|D_\tau\varphi|^2\, d\H-4\gamma_h  \|\nabla v_L\|_{L^{\infty}}\int_{\pa L}\vphi^2\,d\H\,.
\end{split}\]
Note that $\varphi\in T^\perp(\pa L)$ if and only if $\varphi\in H^1(\pa L)$ with zero average on each connected component of $\pa L$. Thus, using Poincar\'e inequality on each connected component of $\pa L$ we have that
$$
\pa^2J_h(L)[\varphi]\geq\frac12 \int_{\pa L}|D_\tau\varphi|^2\, d\H\,,
$$
provided that $h$ is large enough, say $h\geq h_0$. By Theorem~\ref{th:l1min}, there exists $\delta>0$ such that 
\begin{equation}\label{contra}
J_{h_0}(F)>J_{h_0}(L), \qquad \text{for all $|F|=|L|$, $0<\alpha(F,L)<\delta$\,.}
\end{equation}
 We claim that the same holds for all $h>h_0$. Indeed, if for some $h>h_0$ and for some set $F$ with $|F|=|L|$ and $0<\alpha(F,L)<\delta$ the above inequality does not hold, then we have 
 \beq\label{contra2}
 J_{h}(F)\leq J_{h}(L)\,.
 \eeq 
 In turn, since $L$ is the unique global minimizer of the perimeter and thus $\P(F)>\P(L)$, we deduce
 $$
\int_{\T}|\nabla v_L|^2\,dx-\int_{\T}|\nabla v_F|^2\,dx>0\,.
 $$
 But then, by \eqref{contra2}, we get
 $$
 \P(F)-\P(L)\leq\gamma_h\biggl(\int_{\T}|\nabla v_L|^2\,dx-\int_{\T}|\nabla v_F|^2\,dx\biggr)<\gamma_{h_0}\biggl(\int_{\T}|\nabla v_L|^2\,dx-\int_{\T}|\nabla v_F|^2\,dx\biggr)\,,
 $$
 which contradicts \eqref{contra}. Thus, we have proved that
$$
J_{h}(F)>J_{h}(L), \qquad \text{for all $|F|=|L|$, $0<\alpha(F,L)<\delta$}
$$
for all $h\geq h_0$.
As $E_h\to L$ in $L^1$, for $h$ large enough we also have $J_h(E_h)>J_h(L)$, which contradicts the minimality of $E_h$.  
\end{proof}
As an immediate consequence of the above theorem we recover the following result, first proved in \cite{ST}.
\begin{corollary}\label{cor:ST}
Let $N=2$. Fix any $m$ such that $|m| <1-\frac{2}{\pi}$ . Then for small $\gamma > 0$, any solution of 
$$
\min\biggl\{P_{\mathbb{T}^2}(E)+\gamma\int_{\mathbb{T}^2}|\nabla v_E|^2\,dx:\, E\subset\mathbb{T}^2,\,|E|=\frac{m+1}2\biggr\}
$$
belongs to $\mathcal L_m
$, that is, it is  lamellar.
\end{corollary}
\begin{proof}
The proof follows from Theorem~\ref{caffe} and from the fact that if  $|m| <1-\frac{2}{\pi}$, then the lamellar sets of $\mathcal{L}_m$ are the unique global minimizers of the periodic isoperimetric problem in $\mathbb{T}^2$, as proved in \cite{HHM} (see also \cite{CS2}).
\end{proof}
The corollary above holds only for $N=2$, where the minimality range of lamellar sets is completely determined. 
For $N=3$, to the best of our knowledge the global (with uniqueness) minimality of
$\mathcal{L}_m$ is known only in the case $m=0$ (see \cite{H}). In the following result we show the result still holds for $m$ sufficiently close to $0$.
\begin{theorem}\label{erice}\
There exists $\e>0$ such that if $m\in (-\e, \e)$ the lamellar sets in $\mathcal{L}_m$ are the unique solutions  to the corresponding periodic  isoperimetric problem in $\mathbb{T}^3$.
\end{theorem}
\begin{proof}
We argue by contradiction assuming that there exist $m_h\to 0$ and a sequence $E_h$ of global minimizers of
\beq\label{erice1}
\min\Bigl\{P_{\mathbb{T}^3}(E): E\subset \mathbb{T}^3\,, |E|=\frac{m_h+1}{2}\Bigr\}
\eeq
such that $E_h\not\in\mathcal{L}_{m_h}$. 
As before, we may assume that  the sequence $E_h$ converges in $L^1$ to a global minimizer of  \eqref{erice1} with $m_h$ replaced by $0$. By the result of \cite{H} (see also \cite[Theorem 4.3]{CS2}) we may assume that $E_h\to L=\mathbb{T}^2\times[0,\frac12]$ in $L^1$ (up to translation and relabeling of coordinates). 
Moreover, arguing as in Proposition~\ref{prop:EF}, one can show that there exists 
$\lambda>0$ independent of $h$ such that 
each $E_h$ is also a minimizer of the unconstrained penalized problem
$$
\min\Bigl\{P_{\mathbb{T}^3}(E)+\lambda\Bigl||E|-\frac{m_h+1}{2}\Bigr|: E\subset \mathbb{T}^3\Bigr\}\,.
$$
Thus, in particular, all sets $E_h$ are $\omega$-minimizers of the perimeter with the same $\omega=\frac{3\lambda}{4\pi}$. Therefore, by Theorem~\ref{th:white} we deduce that for $h$ large $\pa E_h=\{x+\psi_h(x)e_3:\, x\in \pa L\}$, with $\psi_h\to 0$ in $C^{1}(\pa L)$. By adding the constant $-\frac{m_h}2$ to $\psi_h$ only on the upper part of $\pa L$, we obtain the boundary of a new set $E'_h$ with the same perimeter as $E_h$ and volume $\frac12$. But then, 
$$P_{\mathbb{T}^3}(E_h)=P_{\mathbb{T}^3}(E'_h)> P_{\mathbb{T}^3}(L)\,,$$
by the minimality of $L$.  On the other hand, using the fact that all strips have the same perimeter, $P_{\mathbb{T}^3}(E_h)\leq P_{\mathbb{T}^3}(L)$ by the minimality of $E_h$. This contradiction  concludes the proof.
\end{proof}
\begin{remark}\label{emi}
Note that the argument used in the proof of the previous theorem shows that in any dimension the values of $m$ such that the corresponding strip is the unique minimizer of the perimeter form an open set.
\end{remark}
As before we have the following corollary.
\begin{corollary}\label{cor:nostro}
Let $N=3$. There exists $m_0>0$ and $\gamma_0$ such that for  $|m| <m_0$ such that  any solution of 
$$
\min\biggl\{P_{\mathbb{T}^3}(E)+\gamma\int_{\mathbb{T}^3}|\nabla v_E|^2\,dx:\, E\subset\mathbb{T}^3,\,|E|=\frac{m+1}2\biggr\}
$$
belongs to $\mathcal L_m
$,  provided that $\gamma\leq \gamma_0$.
\end{corollary}
\begin{proof}
The result follows immediately from Theorem~\ref{erice} arguing as in the proof of Theorem~\ref{caffe}. 
\end{proof}

We now conclude this section with a result concerning the local minimality of lamellar configurations with multiple strips.
To this aim, given $m\in (-1,1)$ and an integer $k>1$, we set
 $L_k:=\mathbb T^{N-1}\times\cup_{i=1}^k[\frac{i-1}k, \frac{i-1}k+\frac{m+1}{2k}]$ and denote by $\mathcal L_{m,k}$ the collection of all sets which may be obtained from $L_k$ by translations and relabeling of coordinates.
\begin{proposition}\label{digest}
Fix $m\in (-1,1)$ and $\gamma>0$. Then there exists an integer $k_0$ such that for $k\geq k_0$
all sets in  $\mathcal L_{m,k}$ are isolated local minimizers of \eqref{J}, according to Definition~\ref{def:locmin}.
\end{proposition}
\begin{proof}
First, observe that each $L_k$ is a critical point for $J$.
By Theorem~\ref{th:l1min}, it is enough to show that for $k$ large enough 
$\pa^2J(L_{k})[\vphi]>0$ for all $\vphi\in T^{\perp}(\pa L_k)$.  In fact, by an argument of \cite{MS} it is enough to consider $\varphi\in H^1(\pa L_k)$ with zero average on each connected component of $\pa L_k$. Then 
\[\begin{split}
\pa^2J(L_k)[\varphi]&= \int_{\pa L_k}|D_\tau\varphi|^2\, d\H+8\gamma \int_{\pa L_k}  \int_{\pa L_k}G(x,y)\varphi(x)\varphi(y)d\H(x)\,d\H(y) \\
& \quad +4\gamma\int_{\pa L_k}\partial_{\nu}v\varphi^2\,d\H\\
&\geq \int_{\pa L_k}|D_\tau\varphi|^2\, d\H-4\gamma  \|\nabla v_{L_k}\|_{L^{\infty}}\int_{\pa L_k}\vphi^2\,d\H\,.
\end{split}\]
Note that $v_{L_k}(x)=v_{L_k}(x_N)=\frac{1}{k^2}v_L(k x_N)$ and thus 
 $\|\nabla v_{L_k}\|_{L^{\infty}}=\frac{C}{k}$. Hence, the result follows  using  Poincar\'e inequality. 
\end{proof}

\section{The Neumann problem}\label{secneu}
A variant of our result, which is important in the applications, is the Neumann problem: as before we consider the functional
\beq\label{N.1}
J_N(E):=P_\Omega(E)+\gamma\int_{\Omega}|\nabla v_E|^2\, dx
\eeq
and the function
$$u_E=\chi_E-\chi_{\Omega\setminus E}\;,\qquad m=\medint_\Omega u_E\,dx\;,$$
but the condition on $v_E$ is now
$$\begin{cases}
-\Delta v_E=u_E-m\quad \text{ in }\Omega\\
 \displaystyle \int_\Omega v_E\,dx=0\;,\quad \frac{\pa v_E}{\pa\nu}=0\;,\quad \text{on } \pa\Omega \;.\end{cases}$$
As in \eqref{G1} we have
$$\int_\Omega |\nabla v_E|^2\, dx=\int_\Omega\int_\Om G(x,y)u_E(x)u_E(y)\, dxdy\,,$$
where $G$ is the solution of
$$\begin{cases}
-\Delta_y G(x,y)=\delta_x -\frac1{|\Om|}\quad \text{ in }\Omega\\
 \displaystyle \int_\Omega G(x,y)\,dy=0\;,\quad \nabla_y G(x,y)\cdot \nu(y)=0\;,\quad \text{if } y\in\pa\Omega \;.\end{cases}$$

As in the periodic case, if $E$ is a sufficiently smooth (local) minimizer of the functional \eqref{N.1}, then it satisfies
 the Euler-Lagrange equation 
 $$
 H_{\pa E}(x)+4\gamma v_E(x)=\lambda \qquad\text{for all $x\in \partial E\cap \Om$,}
 $$
and moreover $\pa E$ must meet $\pa \Om$ ortoghonally (if at all), see \cite[Remark 2.8]{CS}.
However, in this paper we shall only deal with case $\pa E\cap \pa \Om=\emptyset$. 
We shall refer to any sufficiently smooth set satisfying the Euler-Lagrange equation  as \emph{a regular critical set} for the functional \eqref{N.1}.

Note that, unlike in the periodic case, the functional $J_N$ is not translation invariant, therefore we don't need to consider the distance $\alpha$ defined in \eqref{recente1}. Precisely, we say that a set $E\CC\Om $ is a \emph{local minimizer} if there exists $\de\geq 0$ such that 
$$
J_N(F)\geq J_N(E)\qquad\text{for all $F\subset \Om$, $|F|=|E|$, and $|E\triangle F|\leq \de$.}
$$ 
If the inequality is strict whenever $|E\triangle F|>0$, we say that $E$ is an \emph{isolated local minimizer}.

 Provided that $\pa E$ does not meet $\pa \Om$,  the regularity result stated in  Theorem~\ref{th:regularity} for the periodic case still holds in the Neumann case, without any change 
in the proof; similarly, if $\Phi:\Om\times(-1,1)\to\Om$ is a  $C^2$-flow, defining the  second variation of $J_N$ at $E$ as in Section~\ref{minlocsecvar}, the representation formula in Theorem~\ref{th:J''} holds. Therefore, we consider the same quadratic form $\pa^2 J_N(E)$ defined in \eqref{pa2J}. 

Since the problem is not translation invariant anymore, the spaces $T(\pa E)$, $T^\perp (\pa E)$, and the decomposition \eqref{tperp} are no longer needed. Therefore, we say that $J_N$ has \emph{positive second variation} at the critical set $E$ if 
$$
\pa^2J_N(E)[\vphi]>0\qquad\text{for all $\vphi\in \Ht(\pa E)\setminus \{0\}$. }
$$
As in Lemma~\ref{coerc}, we immediately have that 
\beq\label{c0neu}
m_0:=\inf\Bigl\{\pa^2J_N(E)[\vphi]:\, \vphi\in\Ht(\pa E)\,, \|\vphi\|_{H^1}=1\Bigr\}>0\,.
\eeq
We now state the main result of the section.
\begin{theorem}\label{th:l1minneu}
Let $E\CC\Om$ be a regular  critical set  with positive second variation. Then there exist $C$, $\delta>0$ such that 
$$
J_N(F)\geq J_N(E)+C|E \triangle F|^2\,,
$$
for all $F\subset\Om$, with   $|F|=|E|$ and $|E \triangle F|<\delta$.
 \end{theorem}
The proof of the result is very similar to the one of Theorem~\ref{th:l1min} with several simplifications due to the fact that $J_N$ is not translation invariant. We give only an outline of the proof, indicating  the main changes.
As in the periodic case, we start  by  a local minimality result with respect to  small $W^{2,p}$ perturbations. More precisely, 
we have the following counterpart to Theorem~\ref{th:c2min}: 
\begin{theorem}\label{th:c2minneu}
Let  $p>\max\{2, N-1\}$ and let $E\CC\Om$ be a regular  critical set for $J_N$  with positive second variation. Then there exist $\delta>0$, $C_0>0$ such that 
$$
J_N(F)\geq J_N(E)+C_0|E \triangle F|^2\,,
$$
whenever  $F\subset \T$ satisifes $|F|=|E|$ and $\pa F=\{x+\psi(x)\nu(x):\, x\in \pa E\}$ for some $\psi$ with $\|\psi\|_{W^{2,p}(\pa E)}\leq \delta$.   
 \end{theorem}
\begin{proof}[Sketch of the proof.] Since the functional is not translation invariant we don't need Lemma~\ref{lm:italiano},  and
inequality \eqref{c2unif} proved in Step 1 of the proof of Theorem~\ref{th:c2min} simplifies to 
$$
\inf\Bigl\{\pa^2J_N(F)[\vphi]:\, \vphi\in \Ht(\pa F)\,, \|\vphi\|_{H^1(\pa F)}=1\Bigr\}\geq\frac{m_0}2\,,
$$
where $m_0$ is the constant defined in \eqref{c0neu}. The proof of this inequality goes exactly as before.

Coming to Step 2 of the proof of Theorem~\ref{th:c2min}, we don't need \eqref{s20} and thus we don't need to replace $F$ by  a suitable translated $F-\sigma$. Instead, we only need to observe that \eqref{s22} is still satisfied and that the rest of the proof remains unchanged.
\end{proof}

Next we observe that the statement of Theorem~\ref{th:inftymin} and its proof remain unchanged, provided  we choose $\delta_0$ such that $\mathcal{I}_{\de_0}(E)\CC\Om$.

We are now left with the  last step of the proof of Theorem~\ref{th:l1minneu}.
\begin{proof}[Proof of Theorem~\ref{th:l1minneu}.] 
 We want to pass from the $L^\infty$-local minimality to the $L^1$-local minimality.
This can be done arguing by contradiction as in the proof of  Theorem~\ref{th:l1min}, by assuming that 
 there exists a sequence $E_h\subset\Om$, $|E_h|=|E|$, such that $\e_h:=|E_h\triangle E|\to 0$ and 
$$
J_N(E_h)\leq J_N(E)+\frac{C_0}{4}|E_h\triangle E|^2\,,
$$
where $C_0$ is as in the statement of Theorem~\ref{th:c2minneu}. 
Then, one replaces the sequence
$E_h$ with a sequence of minimizers $F_h$ of the  penalized functionals
$$
J_N(F)+\Lambda_1\sqrt{\bigl(|F\triangle E|-\e_h\bigr)^2+\e_h}+\Lambda_2\bigl||F|-|E|\bigr|\,,
$$
for suitable   $\Lambda_1$ and $\Lambda_2$ chosen as in the proof of Theorem~\ref{th:l1min}. 

Exactly as before,  one can prove that $\chi_{F_h}\to \chi_{E}$ in $L^1$. An obvious modification of the argument used to prove 
\eqref{infty5} shows that the sets $F_h$ are $\Lambda$-minimizers of the area functional for some $\Lambda>0$ independent of $h$.
Then, Theorem~\ref{th:white} yields that $F_h\to E$ in $C^{1,\alpha}$ for all $\alpha\in (0, \frac12)$. In particular, 
$F_h\CC\Om$ for $h$ large enough. The rest of the proof goes unchanged.
\end{proof}

As  in the previous section we may consider the Ohta-Kawasaki energy, rewritten from \eqref{in.OK} in terms of $u,v$ as
$$
\mathcal E_\e(u)=\e\int_{\Om}|\nabla u|^2\, dx+\frac1\e\int_\Om(u^2-1)^2\, dx
+\gamma_0\int_\Om |\nabla v|^2\,dx$$
where $\gamma_0=16\gamma/3$ and $v$ is the solution to 
 $$
 \begin{cases}
-\Delta v=u-m\quad \text{ in }\Omega\\
 \displaystyle \int_\Omega v\,dx=0\;,\quad \frac{\pa v}{\pa\nu}=0\;,\quad \text{on } \pa\Omega \;.
 \end{cases}
 $$
Fix $m\in (-1, 1)$. We say that a function $u\in H^1(\Om)$ is an \emph{isolated local minimizer} for the functional $\mathcal E_\e$ with prescribed volume fraction $m$, if $\medintinrigo_\Om u\, dx =m$ and there exists $\de>0$ such
that 
$$
\mathcal E_\e(u)<\mathcal E_\e(w) \qquad \text{for all $w\in H^1(\Om)$ s.t. }\medint_\Om w\, dx=m\quad\text{and}\quad 
0<\|u-w\|_{L^1(\Om)}\leq \de\,.
$$
Using \cite[Theorem 2.1]{KS} in place of \cite[Proposition 3.2]{CS2} and the minimality result of Theorem~\ref{th:l1minneu} we get:
\begin{theorem}\label{th:OKneu}
Let $E$ be a regular  critical set  for the functional $J_N$ with positive second variation and $u_E=\chi_E-\chi_{\Om\setminus E}$. 
Then there exist  $\e_0>0$ and a family $\{u_\e\}_{\e<\e_0}$ of isolated local minimizers of $\mathcal E_\e$ with prescribed volume 
$m=\medintinrigo_\Om u_E\, dx$ such that 
$u_\e\to u_E$ in $L^1(\Om)$ as $\e\to 0$.
\end{theorem}

\begin{remark}[Existence of many droplet  stable configurations for the Ohta-Kawasaki energy in two and three dimensions]
Let $\Om$ be a bounded smooth open set  in $\R^3$. In \cite[Theorems 2.1 and 2.2]{RW4} the authors construct a stable critical configuration for the sharp interface functional \eqref{N.1} with a many droplet pattern. More precisely, for any $k\in \N$  they show the existence of a parameter  range
 for $\gamma$ (see \cite[equation (2.12)]{RW4}) such that $J_N$ admits critical set, with positive second variation, made up of
 $k$ connected components that are close to small balls  compactly contained in $\Om$. Theorem~\ref{th:OKneu} then applies and yields the existence of isolated local minimizers of the Ohta-Kawasaki energy $\mathcal E_\e$ with a many droplet pattern for small values of $\e$. A similar two-dimensional construction, to which  Theorem~\ref{th:OKneu} applies as well, has been carried out in \cite{RW3}.
\end{remark}

\section{Appendix}
This section is devoted to the computation of the second variation and to the proof of two auxiliary results.
\begin{proof}[Proof of Theorem~\ref{th:J''}]
We set 
\begin{equation}\label{EF}
E(t):=P_\T(E_t)\quad\text{and}\quad F(t):=\int_\T|\nabla v_t|^2\, dx\,. 
\end{equation}
Arguing as in Step 3 of the proof of \cite[Proposition 3.9]{CMM} (see also \cite[Section 9]{Simon}) we get
\begin{equation}\label{E''}
\begin{split}
E''(0)&= \int_{\pa E}\Bigl(|D_\tau(X\cdot \nu)|^2-|B_{\pa E}|^2(X\cdot \nu)^2\Bigr)\, d\H\\
&+\int_{\pa E}H\Bigl(H(X\cdot \nu)^2+Z\cdot\nu-2X_\tau\cdot D_\tau(X\cdot\nu)+
B_{\pa E}[X_\tau, X_\tau]\Bigr)\, d\H\,,
\end{split}
\end{equation}
where $Z:=\displaystyle\frac{\partial^2 \Phi}{\partial t^2}(x,0)=DX[X]$ and $H$ stands for $H_{\pa E}$.
On the other hand, by \cite[equation (2.67)]{CS}
\begin{equation}\label{F''}
\begin{split}
F''(0)=&8\gamma \int_{\pa E}  \int_{\pa E}G(x,y)\bigl(X\cdot \nu\bigl)(x)\bigl(X\cdot \nu\bigr)(y)d\H(x)\,d\H(y)\\
&+4\gamma\int_{\pa E}\Div(v X)(X\cdot \nu)\, d\H\,.
\end{split}
\end{equation}
Note that
\begin{equation}\label{J''1}
\begin{split}
&\int_{\pa E}\Div(vX)(X\cdot \nu)\, d\H = \int_{\pa E} \nabla v\cdot X(X\cdot \nu)\, d\H+\int_{\pa E}v(\Div X)X\cdot\nu\, d\H\\
&\quad=\int_{\pa E}\pa_\nu v(X\cdot\nu)^2\, d\H+\int_{\pa E}\bigl(D_\tau v \cdot X_\tau\bigr) X\cdot\nu\, d\H+
\int_{\pa E}v(\Div X)X\cdot\nu\, d\H\\
&\quad=\int_{\pa E}\pa_\nu v(X\cdot\nu)^2\, d\H-\int_{\pa E}v\Div_\tau\bigl(X_\tau(X\cdot \nu)\bigr)\, d\H+
\int_{\pa E}v(\Div X)X\cdot\nu\, d\H\,,
\end{split}
\end{equation}
where in the last equality we integrated by parts.
Finally, we claim that 
\begin{equation}\label{claimJ''}
H(X\cdot \nu)^2+Z\cdot\nu-2X_\tau\cdot D_\tau(X\cdot\nu)+
B_{\pa E}[X_\tau, X_\tau]=(X\cdot \nu)\Div X-\Div_\tau\bigl(X_\tau(X\cdot \nu)\bigr)\,.
\end{equation}
Notice that the thesis follows from \eqref{EF}--\eqref{claimJ''}. Hence, it remains to show that \eqref{claimJ''} holds.
To this aim, we observe that 
\[
\begin{split}
X_\tau\cdot D_\tau(X\cdot\nu)&= X_\tau \cdot D(X\cdot\nu)=\nu\cdot DX[X_\tau]+
X\cdot D\nu[X_\tau]\\
&= \nu\cdot DX[X_\tau]+
X_\tau\cdot D\nu[X_\tau]= \nu\cdot DX[X_\tau]+B_{\pa E}[X_\tau,X_\tau]\,.
\end{split}
\]
Therefore, from the last equality, recalling that $Z=DX[X]$, we have that
\begin{equation}\label{J''2}
\begin{split}
H(X\cdot \nu)^2+Z\cdot\nu&-2X_\tau\cdot D_\tau(X\cdot\nu)+
B_{\pa E}[X_\tau, X_\tau]\\
&=H(X\cdot \nu)^2+\nu\cdot DX[X]-X_\tau\cdot D_\tau(X\cdot\nu)-
\nu\cdot DX[X_\tau]\\
&=H(X\cdot \nu)^2+\nu\cdot DX[(X\cdot\nu)\nu]-X_\tau\cdot D_\tau(X\cdot\nu)\\
&=H(X\cdot \nu)^2+(X\cdot\nu)(\Div X-\Div_\tau X)-X_\tau\cdot D_\tau(X\cdot\nu)\,.
\end{split}\end{equation}
On the other hand,
\[
\begin{split}
\Div_\tau\bigl(X_\tau(X\cdot\nu)\bigr)&= X_\tau\cdot D_\tau(X\cdot\nu)+(X\cdot\nu)\Div_\tau X_\tau\\
&= X_\tau\cdot D_\tau(X\cdot\nu)+(X\cdot\nu)\Div_\tau X-(X\cdot\nu)\Div_\tau \bigl((X\cdot\nu)\nu\bigr)\\
&= X_\tau\cdot  D_\tau(X\cdot\nu)+(X\cdot\nu)\Div_\tau X-H(X\cdot\nu)^2\,.
\end{split}
\]
Thus, claim \eqref{claimJ''} follows combining the above identity with \eqref{J''2}.
\end{proof}

\begin{lemma}\label{lm:contazzi}
Under the assumptions of {\rm Theorem~\ref{th:c2min}}, given $q\geq 1$, there exist $\delta$, $C>0$ such that if 
$\|\psi\|_{W^{2,p}(\pa E)}\leq \delta$, then for all $t\in [0,1]$
\beq\label{contazzi0}
\| X\|_{L^q(\pa E_t)}\leq C \|X\cdot \nu^{E_t}\|_{L^q(\pa  E_t)}\,, \qquad
\|D_{\tau_t}\bigl( X\bigr)\|_{L^2(\pa E_t)}\leq C \|  X\cdot \nu^{  E_t}\|_{H^1(\pa  E_t)}\,,
\eeq
where $X$ is defined by \eqref{connect6}.
\end{lemma}
\begin{proof}
Given $\e>0$, by \eqref{campo0} it follows that there exists $\delta>0$ such that if $\|\psi\|_{W^{2,p}(\pa E)}\leq \delta$, then 
$$
\|\nu-\nu^{  E_t}\|_{L^{\infty}(\pa   E_t)}\leq \e\,.
$$
Moreover, setting $\nu_\sigma:=\frac{\nabla d_\sigma}{|\nabla d_\sigma|}$, by \eqref{connect7} we have that $X=(X\cdot\nu_{\sigma})\nu_\sigma$. Moreover, 
from the above inequality and \eqref{connectone}, we get
\beq\label{contazzi1}
\|\nu_\sigma-\nu^{  E_t}\|_{L^{\infty}(\pa   E_t)}\leq 2\e\,.
\eeq
Thus, we have
\begin{align}
\bigl|X_{\tau_t}\bigr|& = \bigl|X-(X\cdot\nu^{  E_t})\nu^{  E_t}\bigr|= \bigl|(X\cdot\nu_\sigma
)\nu_\sigma
-(X\cdot\nu^{  E_t})\nu^{  E_t}\bigr|
\nonumber\\
&\leq \bigl|(   X\cdot\nu^{  E_t})(\nu_\sigma-\nu^{  E_t})\bigr|+
\bigl|   X\cdot (\nu_\sigma-\nu^{  E_t})\nu_\sigma\bigr|
\leq 4\e|   X|\,.\label{contazzi2}
\end{align}
Hence, the first inequality in \eqref{contazzi0} follows.

We now prove the second estimate in \eqref{contazzi0}. Recalling   \eqref{contazzi1}, we have
\begin{align*}
\bigl|D_{\tau_t} X_{\tau_t}\bigr|& = \bigl|D_{\tau_t}X-D_{\tau_t}\bigl((X\cdot\nu^{  E_t})\nu^{  E_t}\bigr)\bigr| = \bigl|D_{\tau_t}\bigl((X\cdot \nu_\sigma
)\nu_\sigma
\bigr)-D_{\tau_t}\bigl((X\cdot\nu^{  E_t})\nu^{  E_t}\bigr)\bigr|\\
&\leq \bigl|D_{\tau_t}\bigl((   X\cdot \nu^{  E_t})(\nu_\sigma
-\nu^{  E_t})\bigr)\bigr|+
\bigl|D_{\tau_t}\bigl(   X\cdot (\nu_\sigma
-\nu^{  E_t})\nu_\sigma
\bigr)\bigr|\\
& \leq C\e \bigl(\bigl|D_{\tau_t}    X\bigr|+\bigl|D_{\tau_t} \bigl(   X\cdot \nu^{  E_t} \bigr)\bigr| \bigr)+
C|X|\bigl(\bigl|D_{\tau_t}\nu_\sigma
  \bigr|+\bigl|D_{\tau_t}\nu^{  E_t}  \bigr|\bigr)\,.
\end{align*}
From this inequality, taking $\e$ (and in turn $\delta$) sufficiently small, also by \eqref{contazzi2}, we deduce that
$$
\bigl|D_{\tau_t}   X_{\tau_t}\bigr|\leq C \bigl|D_{\tau_t} \bigl(   X\cdot \nu^{  E_t} \bigr)\bigr| +
C|   X\cdot \nu^{  E_t} |\bigl(\bigl|D_{\tau_t}\nu_\sigma
  \bigr|+\bigl|D_{\tau_t}\nu^{  E_t}  \bigr|\bigr)\,.
$$
Integrating this inequality, we obtain
\begin{align*}
\bigl\|D_{\tau_t}   X_{\tau_t}\bigr\|^2_{L^2(\pa   E_t )}& \leq C
\bigl\|D_{\tau_t} \bigl(   X\cdot \nu^{  E_t} \bigr)\bigr\|^2_{L^2(\pa   E_t )}
 +C\int_{\pa   E_t }|   X\cdot \nu^{  E_t}|^2\bigl[\bigl|D_{\tau_t}\nu_\sigma
  \bigr|+\bigl|D_{\tau_t}\nu^{  E_t}  \bigr|\bigr]^2\, d\H\\
& \leq C\bigl\|   X\cdot \nu^{  E_t}\bigr\|^2_{H^1(\pa   E_t )}
+C\bigl\|   X\cdot \nu^{  E_t}\bigr\|^{2}_{L^{\frac{2p}{p-2}}(\pa   E_t )}
\bigl\| \bigl|D_{\tau_t}\nu_\sigma
  \bigr|+\bigl|D_{\tau_t}\nu^{  E_t}  \bigr|\bigr\|^2_{L^p(\pa   E_t )}\\
& \leq C\bigl\|   X\cdot \nu^{  E_t}\bigr\|^2_{H^1(\pa   E_t )}\,,
\end{align*}
where the last inequality follows from the Sobolev Embedding theorem and the assumption $p>\max\{2, N-1\}$.
\end{proof}
The next lemma is a consequence of  the classical $L^p$ elliptic  theory.
\begin{lemma}\label{lemmino}
Let $E$ be a set of class $C^2$ and  let  $E_h$ be a sequence of  sets of class $C^{1, \alpha}$ for some $\alpha\in (0,1)$ such that 
$$
\pa E_h=\{x+\psi_h(x)\nu(x):x\in \pa E\}\,,
$$
where $\psi_h\to 0$ in $C^{1, \alpha}(\pa E)$. Assume also that  $H_{\pa E_h}\in L^p(\pa E_h)$ and that 
\beq\label{lemmino1}
H_{\pa E_h}\bigl(\cdot+\psi_h(\cdot)\nu(\cdot)\bigr)\to H_{\pa E}(\cdot)\qquad\text{in $L^{p}( \pa E)$.}
\eeq
Then $\psi_h\to 0$ in $W^{2,p}(\pa E)$.
\end{lemma}
\begin{proof}
We only sketch the proof: by localization, as in the proof of \eqref{infty6} we may reduce to the case when both $\pa E$ and all $\pa E_h$ are graphs of functions $g,g_h$ in a cylinder 
$C=B'\times(-L,L)$, where $B'\subset\R^{N-1}$ is a ball centered at the origin. From our assumptions we have that $ g\in C^2(B')$ and
$$
g_h\to g\in C^{1,\alpha}(B')\;,\qquad \Div\biggl(\frac{\nabla_{x'} g_h}{\sqrt{1+|\nabla_{x'} g_h|^2}}\biggr)\to \Div\biggl(\frac{\nabla_{x'} g}{\sqrt{1+|\nabla_{x'} g|^2}}\biggr)\quad\text{in}\ L^p(B')\;.
$$
Standard elliptic regularity gives that $g_h$ is bounded in $W^{2,p}$ in a smaller ball $B''$, thus we may carry out the differentiation and using the $C^1$ convergence of $g_h$ to $g$ we are led to
$$A_{ij}(x')\nabla^2_{ij}g_h\to A_{ij}(x')\nabla^2_{ij}g\quad\text{in}\ L^p(B'')\;,$$
where
$$A=I-\frac{\nabla_{x'} g\otimes\nabla_{x'} g}{1+|\nabla_{x'} g|^2}\;.$$
As before, standard elliptic estimates imply the strong (local) convergence in $L^p$ of $\nabla^2_{x'} g_h$.
\end{proof}
\begin{remark}\label{rm:lemmino}
If in the above lemma we replace  \eqref{lemmino1} by
$$
\sup_{h}\|H_{\pa E_h}\|_{L^p(\pa E_h)}<\infty\,,
$$
then the same argument shows  that the functions $\psi_h$ are equibounded in $W^{2,p}(\pa E)$.
\end{remark}

\section*{Acknowledgment}
\noindent This research was supported by the 2008 ERC Advanced Grant in ``Analytic Techniques for
Geometric and Functional Inequalities'',  by the ERC grant 207573 ``Vectorial Problems", and by PRIN 2008 ``Optimal Mass Transportation, Geometric and Functional Inequalities and Applications.''

\end{document}